\documentclass[12pt]{article}
\usepackage{graphicx}
\usepackage{amsmath}
\usepackage{amssymb}
\usepackage{theorem}
\usepackage{enumerate} 
\usepackage{color}
\DeclareUnicodeCharacter{2212}{\textendash}

  \usepackage{hyperref}

\numberwithin{equation}{section}

\newtheorem{thm}{Theorem}[section]
\newtheorem{lemma}[thm]{Lemma}
\newtheorem{sublemma}{Sublemma}
\newtheorem{prop}[thm]{Proposition}
\newtheorem{cor}[thm]{Corollary}
{\theorembodyfont{\rmfamily}

\newtheorem{rmk}[thm]{Remark}
}

\renewcommand{\thesubsection}{\arabic{section}.\arabic{subsection}}

\newcommand{\qed}{\hfill \mbox{\raggedright \rule{.07in}{.1in}}}
 
\newenvironment{proof}{\vspace{1ex}\noindent{\bf
Proof}\hspace{0.5em}}{\hfill\qed\vspace{1ex}}
\newenvironment{pfof}[1]{\vspace{1ex}\noindent{\bf Proof of
#1}\hspace{0.5em}}{\hfill\qed\vspace{1ex}}

\newcommand{\A}{{\mathbb A}}
\newcommand{\R}{{\mathbb R}}

\newcommand{\C}{{\mathbb C}}
\newcommand{\Z}{{\mathbb Z}}
\newcommand{\D}{{\mathbb D}}

\newcommand{\B}{{\mathcal B}}

\renewcommand{\H}{{\mathbb H}}
\newcommand{\barH}{{\overline{\H}}}

\newcommand{\essinf}{\operatorname{ess\, inf}}

\newcommand{\supp}{\operatorname{supp}}

\renewcommand{\Re}{\operatorname{Re}}

\newcommand{\sign}{\operatorname{sign}}

\newcommand{\eps}{{\epsilon}}
\def\multi {\boldsymbol{s}}
\newcommand{\SMALL}{\textstyle}

\newcommand{\vertiii}[1]{{\left\vert\kern-0.25ex\left\vert\kern-0.25ex\left\vert #1
    \right\vert\kern-0.25ex\right\vert\kern-0.25ex\right\vert}}

\author{Dalia Terhesiu\thanks{Mathematical Institute, Niels Bohrweg 1,
2333 CA Leiden, The Netherlands}
}

\title{Krickeberg mixing for $\Z$-extensions of Gibbs Markov semiflows}

\begin{document}

 \maketitle

\begin{abstract}
We obtain Krickeberg mixing for a class of $\Z$-extensions of Gibbs Markov 
semiflows with roof function and displacement function not in $L^2$, where previous methods fail. 
This is done via a `smooth tail' estimate for the isomorphic suspension flow.
\end{abstract}



\section{Introduction and main results}
\label{sec-mr}
It is known that $\Z$-extensions of suspension flows over  Markov maps (Young towers)  are
used to model, for instance, tubular Lorentz flows. To simplify the dynamical system setting and get across the analysis, here we focus on  $\Z$-extensions of suspension flows over  Gibbs Markov maps.
Roughly,  a Gibbs Markov map is a  uniformly expanding Markov map with big images and good distortion properties; we refer to~\cite[Ch.\ 4]{Aaronson} for a complete definition. 
Let $F:Y\to Y$ be a Gibbs Markov map preserving an ergodic measure $\mu$ with partition $\alpha$.
Throughout we write $(Y,F, \alpha, \mu)$.
Let $r:Y \to \R_+$ be a roof function (called step time
in \cite{Tho16})
and $\phi:Y \to \Z$ a displacement function (called step function in
\cite{Tho16}) so that $r,\phi\in L^1(\mu)$.
Throughout we assume that $r$ is Lipschitz on each $a\in\alpha$,
and that $\phi$ is $\alpha$-measurable with $\int \phi \, d\mu = 0$.
The $\Z$-extension of the suspension flow over $(Y,F)$
is a flow $\psi_t: \Omega \to \Omega$ defined by $\psi_t(y,q,u) =
(y,q,u+t)$
on the space
$$
\Omega := \{ (y,q,u) \in Y \times \Z \times \R_+ : 0 \leq u \leq r(y) \}
/ \sim \qquad (y,q,r(y)) \sim (F(y), q+\phi(y), 0).
$$
This flow preserves the measure $\mu_{\psi} = \mu \times Leb_{\Z} \times
Leb_{\R}$
where $Leb_\Z$ and $Leb_{\R}$ are counting measure and one-dimensional
Lebesgue measure respectively. Moreover, $\mu_{\psi}$ is ergodic because $\mu$ is ergodic, $r$ is
finite $\mu$-a.e.\ and $\int \phi \, d\mu_\psi = 0$.

Since the $\Z$-component is there, the $\psi_t$-invariant measure is infinite, and the form of
mixing we use in this context is due to Krickeberg~\cite{Krickeberg67}.
Mixing for  $\Z$-extensions of  suspension flows over Young towers (in the sense of~\cite{Young98, Young99}) has been obtained in \cite{DN17},
but their assumptions require that $r$ and $\phi$ are $L^2$-functions. The  mixing result in \cite{DN17} is established by proving
a local limit theorem (LLT) for a large class of group extensions of suspension flows.

In this paper we obtain Krickeberg mixing for $\Z$-extensions of  suspension flows $\psi_t$ over Gibbs Markov maps  when $r,\phi$ are not in $L^2$. 
 Theorem~\ref{thm-mr} below gives
Krickeberg mixing~\cite{Krickeberg67} for  a class of $\Z$-extensions of  Gibbs Markov semiflows with $r, \phi\notin L^2(\mu)$,  satisfying assumptions (H0) and (H1)  below. This is done via  a smooth tail
estimate for the reinduced roof function $\tau$ as in Theorem~\ref{th-tailT} below.  The  present arguments used in the proof of Theorem~\ref{th-tailT} build upon~\cite{Tho16}.
Given Theorem~\ref{th-tailT}, the arguments required for the proof of Theorem~\ref{thm-mr} are essentially 
a `translation' of the arguments in~\cite{Gouezel11} in the set-up of~\cite{MT17}.

We recall that
$(\Omega, \psi_t, \mu_{\psi})$ can be modelled as a suspension flow
$(Y^\tau, \Psi_t, \mu^\tau)$ over
$(Y, \tilde F, \mu)$ where the roof function $\tau : Y \to \R_+$
is the first return time to $Y \times \{ 0 \}\times \{ 0 \}$,
$$
Y^\tau := \{ (y,u) \in Y \times \R_+ : 0 \leq u \leq \tau(y) \} / \sim
\qquad (y, \tau(y)) \sim (\tilde F(y), 0),
$$
and $\tilde F$ is such that $\psi_{\tau(y)}(y,0,0) = (\tilde F(y), 0, 0)$.
The flow $\Psi_t:Y^\tau \to Y^\tau$ is then defined as $\Psi_t(y,u) =
(y,u+t)$ modulo identifications.
Let $\mathcal N$ be the iterate of $(y,q) \mapsto (F(y), q + \phi(q))$
needed to return to
$Y \times \{ 0 \}$, then $\tau = \sum_{j=0}^{\mathcal N-1} r \circ F^j$ and $\tilde F=F^{\mathcal N}$.
Throughout, we let $\tilde \alpha = \bigvee_{j=0}^{\mathcal N-1} F^{-j}(\alpha)$
be the partition associated with $\tilde F$.  Since $(Y,F, \alpha, \mu)$ is 
a probability measure preserving Gibbs Markov map, $(Y,\tilde F, \tilde\alpha, \mu)$
is also a probability measure preserving Gibbs Markov map.

As shown in~\cite{Tho16},
under certain assumptions on $r$ and $\phi$, the tail $1/\mu(\tau>t)$ is regularly varying with index 
less than or equal to $1/2$.
To formulate our assumptions, for functions $v$ that are Lipschitz on each $a\in\alpha$, let
$|1_av|_\vartheta= \sup_{x \neq y \in a} |v(x)-v(y)| / d_\vartheta(x,y)$,
where $d_\vartheta(x,y) = \vartheta^{s(x,y)}$ for some $\vartheta\in (0,1)$
and $s(x,y) = \min\{ n :  F^n(x) \text{ and }  F^n(y) \text{ are in different elements of }\alpha\}$ 
is the separation time.
Throughout we assume

\begin{itemize}
\item[(H0)] 
\begin{itemize}
\item[(i)] The roof function $r$ is bounded from below, say  $\inf r\ge 1$, and it is Lipschitz continuous on each $a \in \alpha$ with uniformly bounded
Lipschitz constant. Also, we require that
$\phi:Y\to\Z$ is constant on partition elements with  $\int \phi\, d\mu=0$. 
%$\alpha$-measurable 
\item[(ii)] The observable $(r,\phi): Y \to ( [0,\infty) , \Z)$ is aperiodic.
\end{itemize}
\end{itemize}
In (H0)(ii), we mean that $(r,\phi)$ is aperiodic  if there exists no non-trivial solution to the equation
$e^{ibr+i\theta\phi} v\circ F=v$,  for  $(b,\theta)\in \R \times [-\pi,\pi) \setminus\{(0,0)\}$.
%, where $\R_K=\{b\in\R:|b|<K\}$ 

\begin{itemize}
\item[(H1)]Let $p\in (1,2]$.
We assume that as $t\to\infty$,
\[
\mu(\phi\le -t)=\mu(\phi\ge t)=\ell(t) t^{-p},\quad \mu(r>t)=\ell(t) t^{-p}+O(t^{-\gamma}), \gamma>2,
\]
for some slowly varying\footnote{We recall that a measurable function
$\ell:(0,\infty) \to (0,\infty)$ is slowly varying if $\lim_{x\to\infty}\ell(\lambda x)/\ell(x) =1$ for all $\lambda>0$.} function $\ell$. In the case $p=2$, we do \emph{not} require that $r,\phi\in L^2(\mu)$.
\end{itemize}

We remark that the evenness of the tails is for simplicity of the exposition only. The proofs work equally well if $\mu(\phi\le -t)=C_1\ell(t) t^{-p}(1+o(1))$ and $\mu(\phi\ge t)=C_2\ell(t) t^{-p}(1+o(1))$, for some $C_1, C_2\ge 0$ and $C_1+C_2>0$.
More importantly, the `tail behaviour assumption'  (H1) on $r$ and $\phi$ are very natural in the context of the 
tubular Lorentz flow with infinite horizon, so for $p=2$. It is known that for Lorentz gases, $|r-|\phi||\le sqrt 2$ and more refined asymptotics on the tail of $\phi$ has been established in~\cite[Lemma 4.2]{PT20}.
%But to  treat this type of example, one needs to go beyond the Gibbs Markov scenario,which requires further work.

Given $\ell$ as in (H1),  we define: i) $\ell_p=\ell$ if $p\in (1,2)$ and ii) $\ell_p(y)=2\int_1^y\frac{\ell(x)}{x}\, dx$, when $p=2$. Under (H1), throughout we let  $\ell^*$ be a slowly varying function such that $\ell^*(t)t^{-1/p}$ is the asymptotic inverse of $\ell_p(t) t^{-p}$.

Write $r^* = \int_Y r\, d\mu$.
Under (H0)(i) and  (H1)~\cite[Proposition 1.3 and Proposition 2.7]{Tho16} 
(in fact, the assumption on $r$ there is relaxed to $r\in L^1(\mu)$ and not necessarily bounded from below)
shows that
\begin{equation}
\label{eq-tailtau}
\mu(\tau\ge t)=\frac{p\sin(\pi/p)(r^*)^{1-1/p}}{\Gamma(1/p)}\frac{1}{t^{1-1/p}\, \ell^*(t)}.
\end{equation}
In particular, the index of regular variation for $1/\mu(\tau\ge t)$ is $\beta=1-1/p\le 1/2$.
Improving on the tail estimate of \eqref{eq-tailtau}, we obtain the following `smooth tail' result,
for which we need to go beyond Karamata-like estimates, but instead use arguments resembling
those used in~\cite{Erickson70} and~\cite{MT17}:
\begin{thm}
\label{th-tailT}
Assume (H0) and (H1). Set $\beta=1-1/p$.
Then there exists a constant $c>0$ that depends on $\beta$, $r^*$ and $F$ such that as $t\to\infty$,
\[
\mu(t\le \tau\le t+1)=c t^{-(1+\beta)}\ell^*(t)^{-1}(1+o(1)).
\]
\end{thm}
\begin{rmk} If $r,\phi$ satisfy (H0) and $r,\phi\in L^2$, we do not require any special tail assumption and several steps in the proof of Theorem~\ref{th-tailT} can be considerably simplified.

In the present proofs we do not make any attempt to obtain the precise expression of the constant $c$, though clearly $c$ has to match the precise constant in~\eqref{eq-tailtau}.
\end{rmk}
We mention that  Theorem~\ref{th-tailT}  on the smooth tail of $\tau$ is a result of independent interest
and in this work we use this reult to obtain mixing for the semiflow $\Psi_t$
(and thus mixing for the $\Z$-extensions of  the suspension flow $\psi_t$).

Define $m(t)=\int_0^t \mu(\tau>x)\, dx$ and note that $m(t)$ is regularly varying with index $1-\beta=1/p$ (this follows directly from~\eqref{eq-tailtau}).
With this specified we state

\begin{thm}
\label{thm-mr}Assume (H0) and (H1).
Let $A, B\subset Y$  with $A\in\tilde\alpha$
and $B$ measurable. Let $A_1=A\times[a_1,a_2]$, $B_1=B\times[b_1,b_2]$ with $a_1\le a_2\le \inf_A\tau$ and $b_1\le b_2\le \inf_B\tau$.
%are measurable subsets of $Y^{\tau}$.
Set $d_\beta=\frac{\sin\pi\beta}{\pi}$.
Then
\[
\lim_{t\to\infty} m(t)\mu^{\tau}(A_1\cap \Psi_t^{-1}B_1)  =d_\beta\mu^{\tau}(A_1)\mu^{\tau}(B_1).
\]
\end{thm}

\begin{rmk} We do not need the full strength of Theorem~\ref{th-tailT} in the proof of Theorem~\ref{thm-mr}, but only that 
$\mu(t\le \tau\le t+1)=O(t^{-(1+\beta)}\ell^*(t)^{-1})$.

It might be that for  Theorem~\ref{thm-mr} this big O assumption can be relaxed further
	given recent results of~\cite{CaravennaDoney}
on necessary and sufficient conditions for the asymptotic of renewal sequences with infinite mean.
\end{rmk}

We recall that: a) \cite{DN-PC} obtained mixing for a class of Markov suspension flows with regular variation of index $\beta\in (0, 1)$;
b) \cite{MT17} obtained mixing under mild abstract assumptions for, not necessarily Markov, 
suspension flows with regularly varying tails of roof functions of index in $\beta\in(1/2, 1]$.

Although the mixing result~\cite[Theorem 5.1]{DN-PC} holds for all $\beta\in (0, 1)$, it is explicitly stated in terms of suspension flows over LSV maps (as in~\cite{LiveraniSaussolVaienti99}). It is not clear to us if the argument of~\cite[Proposition 5.3]{DN-PC}, on which~\cite[Theorem 5.1]{DN-PC} relies,
can be easily generalized to Gibbs Markov maps that do not arise from inducing LSV maps to good sets.

Although the mixing result~\cite[Theorem 2.3]{MT17} does not apply here due to the range of $\beta$, the previous big tail result of~\cite{Tho16} 
as recalled in~\eqref{eq-tailtau} together with~\cite[Theorem 2.4]{MT17} ensure a liminf result established, among others, via an LLT for the roof function $\tau$
and the base of the semiflow $(Y,\tilde F)$ as in~\cite[Theorem 2.7]{MT17}. 

%We finish this introduction with some remarks on potential extensions.
We believe that the arguments in this paper can be adjusted to work for $\Z^2$-extensions of Gibbs Markov semiflows.
We also believe that the method can be applied to the infinite horizon tubular Lorentz flow which can be viewed as a $\Z^d$-extension ($d=1,2$)
of a suspension flow over a Young tower with exponential tails (see \cite{SV04} for the treatment of the $\Z^d$-extension over the map). Here we restrict
to $\Z$-extensions of the suspension flows over Gibbs Markov maps.

\medskip {\bf Notation:}
We write $a_n \sim b_n$ if $a_n/b_n \to 1$. We use
``big O'' and $\ll$ interchangeably, writing
$a_n=O(b_n)$ or $a_n\ll b_n$ if there is a constant
$C>0$ such that $a_n\le Cb_n$ for all $n\ge 1$. 
Similarly, $a_n = o(b_n)$ means that $\lim_{n \to \infty} a_n/b_n = 0$.

\section{Strategy and proof of Theorem~\ref{th-tailT}}
\label{sec-stratT1}

By definition, $(Y,\tilde F, \tilde\alpha, \mu)$
is a probability measure preserving Gibbs Markov map.
Let  $R$ be the transfer operator defined by  $\int_Y R v_1 v_2\, d\mu=\int_Y v_1 v_2\circ \tilde F\, d\mu$, $v_1\in L^1(\mu)$, $v_2\in L^\infty(\mu)$.
Let $\hat R(s)v= R(e^{-s\tau}v)$, $s\in\C$ be the perturbation of $R$.

First, we collect some identities. For $u\ge 0$, we define the measures $\nu_u$ on the positive real  line such that
$\frac{d\nu_u}{d\tau_*\mu}(t)=te^{-ut}$; in particular, $\frac{d\nu_0}{d\tau_*\mu}(t)=t$. 
With these defined we see that
\begin{align}
\label{eq-nurel}
t\mu(t\le \tau \le t+1)\le \nu_0([t,t+1])\le (t+1)\mu(t \le \tau \le t+1).
\end{align}
Hence,
\begin{align}
\label{eq-smalltder}
t\mu(t\le \tau\le t+1)=\nu_0([t,t+1])+e(t),
\end{align}
where $e(t)=O(\mu(t\le \tau\le t+1))$. 
 Note that for $s=u-ib$, $u\ge 0$ and $b\in\R$,
\begin{align}
\label{eq-tauR}
 \int_0^\infty e^{-st}\, d\tau_*\mu=\int_Y e^{-s\tau}\, d\mu=\int_Y\hat R(s)1\, d\mu.
\end{align}
For $u>0$, using the definition of $\nu_u$ for the first equality and  differentiating in $b$ for the second gives
\begin{align}
\label{eq-smt2}
\int_0^\infty  e^{-st}\, d\nu_0(t)=\int_0^\infty  e^{ibt}\, d\nu_u(t)=-i\frac{d}{db}\Big(\int_Y \hat R(s) 1, d\mu\Big)
=A(s).
\end{align}

By~\eqref{eq-nurel},  $\nu_0([0, L])=\int_0^L t d\tau_*\mu(t)\ge \sum_{j=0}^{L-1} j\mu(j\le \tau\le j+1)\ge (L-1)\mu(\tau\ge L)$. This  together with~\eqref{eq-tailtau}
implies that $\nu_0([0, L])$ grows like $L^{1/p}\ell^*(L)^{-1}$ which goes to $\infty$ as $L\to\infty$. So,
$\nu_0$ is an infinite measure.

Our strategy for obtaining the asymptotics of $\mu(t \le \tau \le t+1)$, as $t\to\infty$
stated in Theorem~\ref{th-tailT}
is to use
an analogue of~\cite[Inversion formula, Section 4]{Erickson70} obtained in~\cite[Proposition 4.1]{MT17} (for different purposes recalled in Section~\ref{sec-abstrsetup}).
The key new ingredient required to apply this strategy to the present set-up is
Proposition~\ref{prop:Fga} below; its proof is postponed to Section~\ref{sec-a}.
To state this result, we need more terminology.

For each $a>0$, let $\hat g_a(0)=1$
and for $x\neq 0$, define
\begin{align}
\label{eq-hatga}
\hat g_a(x)=\frac{2(1-\cos ax)}{a^2x^2}
\end{align}
and note that
$\hat g_a$ is the Fourier transform of
\begin{align}\label{eq-Fourga}
g_a(b)=
\begin{cases}
a^{-1}(1-|b|/a), &  |b|\leq a, \\
0, & |b|>a.
\end{cases}
\end{align}

\begin{prop} \label{prop:Fga} Let $\zeta(t)=t^{1-1/p}\ell^*(t)^{-1}$.
 For all $a>0$ and $\lambda\in\R$,
\begin{align*}
\lim_{t\to\infty} \zeta(t)\int_{-\infty}^\infty e^{-itb} g_a(b+\lambda) A(-ib)\, db=\pi d_p g_a(\lambda)\in\R,
\end{align*}
where $d_p$ is a positive constant that depends only on $p$ and $F$.
\end{prop}
%We  start with an analogue~\cite[Inversion formula, Section 4]{Erickson70} (with proof in Section~\ref{sec-a}).

Given Proposition~\ref{prop:Fga}, the proof of Theorem~\ref{th-tailT} below is similar to the argument used in the proof of~\cite[Theorem 2.3]{MT17}.
Since it is short, we provide the complete proof along with the auxiliary results.
Given $V(x):=V([0,x])=\frac{1}{2}(\nu_0([0,x])+\nu_0(-[0,x]))$ (with $\nu_0(-I)=\nu(\{x: -x\in I\})$) we have
\begin{prop}{\bf{~\cite[Proposition 4.1]{MT17}.}}
\label{prop:inv}
Let $g:\C\to\C$ be a continuous compactly supported function
with Fourier transform
$\hat g(x)=\int_{-\infty}^\infty e^{ixb} g(ib)\, db$
satisfying
$|\hat g(x)|=O(x^{-2})$ as $x\to\pm\infty$.
Then for all $\lambda,t\in\R$,
\begin{align*}
\int_{-\infty}^\infty e^{-i\lambda(x-t)}\hat g(x-t)\,dV(x)
=&\lim_{u\to 0}\int_{-\infty}^\infty e^{-itb} g(b+\lambda)\Re A(u-ib)\, db\\
=&\int_{-\infty}^\infty e^{-itb} g(b+\lambda)\Re A(-ib)\, db.
\end{align*}
\end{prop}

\begin{prop}{~\bf{\cite[Lemma 8]{Erickson70}}}
\label{prop:Erickson}
Let $\{\mu_t,\,t>0\}$ be a family of measures such that $\mu_t(I)<\infty$ for every
compact set $I$ and all $t$. Suppose that for some constant $C$,
\begin{align*}
\lim_{t\to\infty}\int_{-\infty}^{\infty} e^{-i\lambda x}\hat g_a(x)\, d\mu_t(x)=
C\int_{-\infty}^{\infty} e^{-i\lambda x}\hat g_a(x)\, dx,
\end{align*}
for all $a>0$, $\lambda\in\R$. Then $\mu_t(I)\to C |I|$
for every bounded interval $I$, where $|I|$ denotes the length of $I$. \qed
\end{prop}

\begin{pfof}{Theorem~\ref{th-tailT}}
With the convention $I+t=\{x:x-t\in I\}$, let  
\[
\mu_t(I)=2\zeta(t)V(I+t)=\zeta(t)(\nu_0(I+t)+\nu_0(-I-t))\]
and note that $\zeta(t)\nu([t,t+1])=\mu_t([0,1])$.
Now,
\begin{align*}
\zeta(t)\int_{-\infty}^\infty e^{-i\lambda(x-t)}\hat g_a(x-t)\,dV(x)&=
\zeta(t)\int_{-\infty}^\infty e^{-i\lambda x}\hat g_a(x)\,dV(x+t)\\
&=\frac12\int_{-\infty}^{\infty} e^{-i\lambda x}\hat g_a(x)\, d\mu_t(x).
\end{align*}
Since $\hat g_a$ satisfies the assumptions of Proposition~\ref{prop:inv},
\begin{align*}
\int_{-\infty}^{\infty} e^{-i\lambda x}\hat g_a(x)\, d\mu_t(x)=2\zeta(t)\int_{-\infty}^\infty e^{-itb} g_a(b+\lambda)\Re  A(-ib) \, db.
\end{align*}
By Proposition~\ref{prop:Fga} together with the
Fourier inversion formula
$\int_{-\infty}^{\infty} e^{-i\lambda x}\hat g_a(x)\, dx=2\pi g_a(i\lambda)$,
\begin{align*}
\lim_{t\to\infty}\int_{-\infty}^{\infty} e^{-i\lambda x}\hat g_a(x)\, d\mu_t(x)
= 2\pi d_p g_a(\lambda) =d_p \int_{-\infty}^{\infty} e^{-i\lambda x}\hat g_a(x)\, dx.
\end{align*}
Hence, the hypothesis of Proposition~\ref{prop:Erickson} holds  with $C=d_p$.
It follows from Proposition~\ref{prop:Erickson} with $I=[0,1]$ that
\(
\zeta(t)\nu([t,t+1])=\mu_t([0,1])\to d_p,
\)
as $t\to\infty$. The conclusion follows from this together with~\eqref{eq-smalltder} and the fact that $\zeta(t)=t^{1-1/p}\ell^*(t)^{-1}$.~\end{pfof}

\section{Asymptotics of $A(u-ib)$ as $u,b\to 0$ and proof of Proposition~\ref {prop:Fga} }
\label{sec-a}

An essential ingredient for the proof of Proposition~\ref {prop:Fga} 
is Lemma~\ref{lemma-derS} below, which gives the asymptotic behaviour of $A(u-ib)$ as $u,b\to 0$.
Before its statement, we briefly explain the strategy of proof.
% in~\cite{Tho16} for obtaining the asymptotic of $\mu(\tau>t)$ (such as~\eqref{eq-tailtau})and provide the main ingredients required in the statement and proof of  Lemma~\ref{lemma-derS}. 
The key observation in~\cite{Tho16} to obtain~\eqref{eq-tailtau} (also to be  exploited here)
is that the perturbed transfer operator $\hat R(u-ib)$ associated with $\tilde F$
can be understood via a double perturbation
of the transfer operator for $F$, which we denote by $L$, perturbed with $r$ and $\phi$. For $u, b \ge 0$
and $\theta\in [-\pi,\pi)$, let
\begin{equation}
\label{eq-opppppp}
\hat L(u-ib, i\theta)v=L(e^{-(u-ib)r}e^{i\theta\phi}v).
\end{equation}
%Recall that $(I-\hat R(u-ib))^{-1}$ is well defined for all $u\ge 0$ and all $b\in\R\setminus\{0\}$.

It is known and recalled below that $\hat L$ has good spectral properties in the Banach space $\B_\vartheta$ with norm $\|.\|_\vartheta$. 
Here, $\B_\vartheta$ is the space of bounded piecewise H{\"o}lder functions; $\B_\vartheta$
is  compactly
embedded in $L^\infty(\mu)$. The norm on $\B$ is defined by $\| v \|_\vartheta = |v|_\vartheta + |v|_\infty$,
where $|v|_\vartheta = \sup_{a \in\alpha} \sup_{x \neq y \in a} |v(x)-v(y)| / d_\vartheta(x,y)$,
where $d_\vartheta(x,y) = \vartheta^{s(x,y)}$ for some $\vartheta\in (0,1)$,
and $s(x,y) = \min\{ n :  F^n(x) \text{ and }  F^n(y) \text{ are in different elements of } \alpha\}$ 
is the separation time.

Under (H0)(i) and (H1), an argument similar to the one used in~\cite[Lemma 2.6]{Tho16} verifies that when viewed as an operator on the Banach space $\B_\vartheta(Y)$,
the spectral radius of $\hat L(u-ib, i\theta)$ is strictly less than $1$ for all $u\ge 0$ and for all $(b,\theta)\in B_\delta(0,0)$ for some $\delta>0$.
By (H0)(ii), the same holds for all $(b,\theta)\in \R \times [-\pi,\pi)\setminus\{(0,0)\}$.
%As clarified in Subsection~\ref{subsec-sppropL}, by (H0), when viewed as an operator on $\B_\vartheta$, $%\hat L(u-ib, i\theta)$ has the property that its spectral radius of  is strictly less than $1$ for 
%all $u\ge 0$ and for all $(b,\theta)\in \R \times [-\pi,\pi)\setminus\{(0,0)\}$. 
Thus, $(I-\hat L(u-ib, i\theta))^{-1}$ is well defined for all $u\ge 0$ and for all 
$(b,\theta)\in \R \times [-\pi,\pi)\setminus\{(0,0)\}$.
By the argument of~\cite[Proof of Lemma 1.8]{Tho16}, for all $v\in \B_\vartheta$, $u\ge 0$
and $b\in \R\setminus\{0\}$,
\begin{align}
\label{eq-rel}
(I-\hat R(u-ib))^{-1}v=\frac{1}{2\pi}\int_{-\pi}^{\pi}(I-\hat L(u-ib, i\theta))^{-1}v\, d\theta.
\end{align}
In particular,  for all $u\ge 0$ and $b\in \R\setminus\{0\}$, the LHS of~\eqref{eq-rel} is well defined.
\begin{rmk}
\label{rmk-sphatr}
For use in Section~\ref{sec-verH}, we note that the spectral radius of $\hat R(u-ib)$ is strictly less than $1$; here, $\hat R$ is viewed as an operator acting on a Banach space
$\B_{\vartheta_0}$, for some $\vartheta_0$, associated with Gibbs Markov $(Y, \tilde F,\tilde\alpha,\mu)$. 
\end{rmk}
Define
\[
S(u-ib):=\int_{-\pi}^{\pi}(I-\hat L(u-ib, i\theta))^{-1}\, d\theta.
\]
Controlling the asymptotics  as $u, b\to 0$ of $S(u-ib)^{-1} 1$  
is the main step in estimating $\mu(\tau>t)$, when combined with~\eqref{eq-tauR}. In fact, as in~\cite{Tho16},
to estimate $\mu(\tau>t)$ it suffices to work with \emph{real Laplace transforms}, that is work with $b=0$ throughout.
For the purpose of estimating the \lq{small tail}\rq\ $\mu(t\le\tau\le t+1)$, here we shall
use~\eqref{eq-rel} to estimate the derivative $\frac{d}{db}\int_Y\hat R(u-ib)1\, d\mu$, 
as $u,b\to 0$ and, thus, the asymptotics of $A(u-ib), b\to 0$ as  $u,b\to 0$ (via~\eqref{eq-smt2}).

We state the precise result on the asymptotics of $A(u-ib)$ below and defer its proof to Section~\ref{subsec-sppropL}. 
Before its statement we recall the following notation: we write 
$B(x)\sim c(x)P$  for bounded operators $B(x), P$ acting on some Banach space $\B$ with 
norm $\|\, \|_{\B}$ if $\|B(x)-c(x) P\|_{\B} = o(c(x))$.
\begin{lemma}
\label{lemma-derS}
Assume (H0) and (H1). Let $\ell^*$ be as in (H1). There exists $\eps_0>0$ 
so that  the following hold for all $u, b\in B_{\eps_0}(0)$.

\begin{itemize}
\item[i)]$\|\frac{d}{db}S(u-ib)^{-1}\|_\vartheta\le C |u-ib|^{-1/p}\ell^*(1/|u-ib|)^{-1}$, for some positive constant $C$.
Also, as $b\to 0$,
$\frac{d}{db}S(-ib)^{-1}\sim i C_{p}|b|^{-1/p}\ell^*(1/|b|)^{-1}P$,
where $C_{p}$ is a complex constant (independent of $b$) with $\Re C_p>0$ and
$P$ is an operator defined by $Pv=\int_Y v\, d\mu$.

\item[ii)] For any $\eps>0$, $\|\frac{d}{db}S(u-ib)^{-1}-\frac{d}{db}S(-ib)^{-1}\|_\vartheta\le C_\eps u^{1-\eps}|u-ib|^{-1/p}\ell^*(1/|u-ib|)^{-1}$
for some positive constant $C_\eps$.

\item[iii)] For any $\eps>0$,
\(
\|\frac{d^2}{db^2}S(u-ib)^{-1}\|_\vartheta\le  C(|u-ib|^{-1/p-\eps} u^{p-2-\eps}+|u-ib|^{-1/p-1-\eps}),
\) for some positive constant $C_\eps$.
\end{itemize}
\end{lemma}

Using~\eqref{eq-rel}, we have
\[
\frac{d}{db}\int_Y \hat R(u-ib) 1\, d\mu=-\frac{1}{2\pi}\frac{d}{db}\Big(\int_Y S(u-ib)^{-1} 1\, d\mu\Big).
\]
Using the definition of $A(s)$ in~\eqref{eq-smt2} with $s=u-ib$, 
%\frac{d}{db}\Big(\frac{1-e^{-s}}{s}\Big)\int_Y \hat R(s) 1\, d\mu+ \frac{1-e^{-s}}{s}  \frac{d}{db}\Big(\int_Y S(s)^{-1} 1\, d\mu\Big).
\[
A(u-ib)=-i\frac{d}{db}\Big(\int_Y S(u-ib)^{-1} 1\, d\mu\Big).
\]
This together with the first part of  Lemma~\ref{lemma-derS} i) implies that
as $u,b\to 0$,
\begin{align}
\label{eq-A1}
|A(u-ib)|\ll |u-ib|^{-1/p}\ell^*(1/|u-ib|)^{-1}.
\end{align}
Also, by the second part of Lemma~\ref{lemma-derS} i), the following holds under (H0) and (H1),
as $b\to 0$,
\begin{align}
\label{eq-A}
A(-ib)= C_p |b|^{-1/p}\ell^*(1/|b|)^{-1}(1+o(1)).
\end{align}
Moreover, by Lemma~\ref{lemma-derS} ii) and iii), for any $\eps>0$,
\begin{align}
\label{eq-a2} 
 |A(u-ib)-A(-ib)|\ll  u^{1-\eps}|u-ib|^{-1/p}\ell^*(1/|u-ib|)^{-1}
\end{align}
and
\begin{align}
\label{eq-A3} 
 |\frac{d}{db}A(u-ib)|\ll |u-ib|^{-1/p-\eps} u^{p-2-\eps}+|u-ib|^{-1/p-1-\eps}.
\end{align}

We now provide the

\begin{pfof}{Proposition~\ref {prop:Fga}} 
Given the definition of $g_a(b)$ in~\eqref{eq-Fourga}, let $\gamma_a(ib)=g_a(b)$.
In order to exploit the differentiability properties of $A(u-ib)$ (inside the  proof of Lemma~\ref{lem:I2} below)
we need an analytic version of $\gamma_a$.

It follows  from the definition that 
$\gamma_a^{+}(s):=\frac{1}{a}\left(1+\frac{is}{a}\right)$ is the analytic extension
of $\gamma_a|_{(0,a)i}$ to $\C$. Similarly, $\gamma_a^{-}(s):=\frac{1}{a}\left(1-\frac{is}{a}\right)$ is the analytic extension
of $\gamma_a|_{(-a,0)i}$ to $\C$.
With this notation, and recalling that $g_a(b) = 0$ for $|b| > a$, we have
\begin{align}
\label{eq-ipm}
&\int_{-\infty}^\infty g_a(b+\lambda) A(-ib)e^{-ibt}\, db\\
\nonumber&=\int_{-\lambda}^{a-\lambda} \gamma_a^{+}(i(b+\lambda)) A(-ib)e^{-ibt}\, db
+\int_{-a-\lambda}^{-\lambda} \gamma_a^{-}(i(b+\lambda)) A(-ib)e^{-ibt}\, db =I^{+}+I^{-}.
\end{align}
By Cauchy's theorem,
\begin{align*}
I^+ =& \int_{-\lambda}^{a-\lambda} \gamma_a^{+}(\frac1t+i(b+\lambda)) A(\frac1t-ib)e^{-ibt}\, db\\
&+\int_0^{\frac1t} \gamma_a^{+}(u) A(u+i\lambda)e^{(u+i\lambda)t}\, du
- \int_0^{\frac1t} \gamma_a^{+}(u+ia) A(u-i(a-\lambda))e^{(u-i(a-\lambda))t}\, du,
\end{align*}
and analogously,
\begin{align*}
I^-=& \int_{-a-\lambda}^{-\lambda}\gamma_a^{-}(\frac1t+i(b+\lambda)) A(\frac1t-ib)e^{-ibt}\, db\\
&-\int_0^{\frac1t} \gamma_a^{-}(u) A(u+i\lambda)e^{(u+i\lambda)t}\, du
+ \int_0^{\frac1t} \gamma_a^{+}(u+ia) A(u-i(a-\lambda))e^{(u-i(a-\lambda))t}\, du,
\end{align*}
By~\eqref{eq-A1},
$\|A(u-i(a-\lambda))\|\ll a^{-p}$. Thus, the last terms of the RHS for $I^+$ and $I^-$ are $O(t^{-1})$ because the integrand is bounded
and the integration path has length $t^{-1}$.

Also, by~\eqref{eq-A1} (with $b=\lambda$),
$\|A(u+i\lambda))\|\ll |\lambda|^{-p}$, for all $\lambda\ne 0$.
Thus, for all $\lambda\ne 0$, the middle terms of the RHS for $I^+$ and $I^-$ are $O(t^{-1})$ because the integrand is bounded
and the integration path has length $t^{-1}$.

Moreover, when $\lambda=0$, we have the desired cancellation in the middle terms of the RHS cancel
when taking the sum $I^+ + I^-$. That is, using the definition of $\gamma_a^\pm$
and again~\eqref{eq-A1} (with $b$=0),
\begin{align*}
 &\left|\int_0^{\frac1t} \gamma_a^{+}(u) A(u)e^{ut}\, du
 -\int_0^{\frac1t} \gamma_a^{-}(u) A(u)e^{ut}\, du\right|
 \le a^{-2}\int_0^{\frac1t} u|A(u)|e^{ut}\, du\\
 &\le a^{-2}\int_0^{\frac1t} u^{1-1/p}\ell^*(1/u)\, du\le Ca^{-2} t^{2-1/p-\eps},
\end{align*}
for some $C>0$ and any $\eps>0$.
Altogether,
\begin{align*}
\int_{-\infty}^\infty g_a(b+\lambda) A(-ib)e^{-ibt} &\, db =
\int_{-\lambda}^{a-\lambda} \gamma_a^{+}(\frac1t+i(b+\lambda)) A(\frac1t-ib)e^{\frac1t-ibt}\, db \\
&+\int_{-a-\lambda}^{-\lambda} \gamma_a^{-}(\frac1t+i(b+\lambda)) A(\frac1t-ib)e^{\frac1t-ibt}
+ O(t^{-1}).
\end{align*}
Next, it follows from the definition that
\begin{align}
\label{eq-gpma}
|\gamma_a^{\pm}(u+ib)-\gamma_a^{\pm}(ib)|=\frac{u}{a^2}.
\end{align} 
Therefore
\begin{align*}
 \left| \int_{-\lambda}^{a-\lambda} \gamma_a^+ (\frac1t + i(b+\lambda)) A(\frac1t - ib) \right.
 & e^{(\frac1t-ib)t} \, db \left. - \int_{-\lambda}^{a-\lambda} \gamma_a^+(i(b+\lambda)) A(\frac1t-ib) e^{(\frac1t-ib)t} \, db \right| \\
\ll &  t^{-1} \int_{-\lambda}^{a-\lambda} |A(\frac1t - ib)| \, db \\
\ll & t^{-1} \int_{-\lambda}^{a-\lambda} |\frac1t -ib|^{-\frac1p} \ell^*( 1/|\frac1t-ib|)^{-1}  \, db \ll t^{-1},
\end{align*}
and a similar estimate holds for the integral over $\gamma^-_a$.
Therefore
\begin{align*}
\int_{-\infty}^\infty g_a(b+\lambda) A(-ib)e^{-ibt} &\, db =
\int_{-\lambda}^{a-\lambda} \gamma_a^{+}(i(b+\lambda)) A(\frac1t-ib)e^{\frac1t-ibt}\, db \\
&+\int_{-a-\lambda}^{-\lambda} \gamma_a^{-}(i(b+\lambda)) A(\frac1t-ib)e^{\frac1t-ibt}
+ O(t^{-1}).
\end{align*}
At this moment, the arguments of $\gamma_a^{\pm}$ are all on the imaginary axis again, with imaginary part $\leq a$,
so we can switch back from $\gamma_a^{\pm}$ to $\gamma_a(ib)=g_a(b)$:
\begin{align*}
\int_{-\infty}^\infty g_a(b+\lambda) A(-ib)e^{-ibt} &\, db =
\int_{-a-\lambda}^{a-\lambda} \gamma_a(i(b+\lambda)) A(\frac1t-ib)e^{\frac1t-ibt}\, db 
+ O(t^{-1}).
\end{align*}
Recall that $\zeta(t)=t^{1-1/p}\ell^*(t)^{-1}$ and that we are interested in $\zeta(t) \int_{-\infty}^\infty g_a(b+\lambda) A(-ib)e^{-ibt} \, db$.
Using the previous displayed equation,
\[\zeta(t) \int_{-\infty}^\infty g_a(b+\lambda) A(-ib)e^{-ibt} \, db 
= \zeta(t) I_1(t,M) + \zeta(t) I_2(t,M) + O(t^{1-q_0 -1/p} \ell^*(1/t)^{-1} ),
\]
for
\[
I_1(t,M) = \int_{|b| < M/t}\gamma_a(i(b+\lambda)) A(\frac1t-ib)e^{\frac1t-ibt}\, db 
\]
(which is in fact zero for large $t$  if $0 \notin [-a-\lambda, a-\lambda]$) and
\[
I_2(t,M) = \int_{-a-\lambda \le b \le a-\lambda, |b|> M/t} \gamma_a(i(b+\lambda)) A(\frac1t-ib)e^{\frac1t-ibt}\, db.
\]

The conclusion of Proposition~\ref{prop:Fga} follows from the estimates
for $I_1(t,M)$ and $I_2(t,M)$ below. 
More precisely, Lemma~\ref{lem:I1} below gives the exact term showing also that
$\lim_{t\to\infty} \zeta(t)I_1(t,M) =\lim_{t\to\infty}\zeta(t)\int_{-M/t}^{M/t} \gamma_a(i(b+\lambda)\Re A(ib)e^{-ibt}\,db$.
Taking $M=t^{1/2}$, we have $\lim_{t\to\infty} \zeta(t)I_1(t,M) =\lim_{t\to\infty}\zeta(t)\int_{-\infty}^{\infty} \gamma_a(i(b+\lambda)\Re A(ib)e^{-ibt}\,db$,
which gives the first equality in the statement.

 Lemma~\ref{lem:I2} with $M=t^{1/2}$ and $\eps < \frac{1}{8p}(p-1)^2$ shows that
$|\zeta(t)I_2(t,M)| \to 0$ 
as $t \to \infty$.
\end{pfof}

\begin{lemma} \label{lem:I1} For any $M>1$,
\[\lim_{t\to\infty}\zeta(t)I_1(t,M) =\lim_{t\to\infty}\zeta(t)\int_{-M/t}^{M/t} g_a(b+\lambda)\Re A(-ib)e^{-ibt}\,db=\pi d_p g_a(\lambda) +q(M), \]
where $d_p$ is a positive constant independent of $M$ and $q(M)\le C M^{-1/p}$, for some $C>0$.
\end{lemma}

\begin{proof}
Throughout this proof we use the same notation as in the proof of Proposition~\ref{prop:Fga}.
It follows from the definition of $\gamma_a$ that $|\gamma_a(ib_1)-\gamma_a(ib_2)|\le a^{-2}|b_1-b_2|$.
Hence
\begin{align}
\label{eq-onemore3}
\nonumber \Big|I_1(t,M)-g_a(i\lambda) I_1^{\pm}(t,M)\Big| & \le \int_{-M/t}^{M/t} |\gamma_a(i(b+\lambda))-\gamma_a(i\lambda)|\,|A(1/t-ib)|\,db \\
& \le 2a^{-2}M t^{-1}\int_0^{M/t}|A(1/t-ib)|\,db.
\end{align}
By \eqref{eq-A}, there exists $\delta>0$ such that for  all $t>M/\delta$,
 \begin{align}
\label{eq-onemore4}
\int_0^{M/t}|A(1/t-ib)|\,db \le \int_0^{M/t}b ^{-1/p} \ell^*(1/b) \,db\le 1.
\end{align}
Next, write
\begin{align*}
\lim_{t\to\infty}\zeta(t)I_1(t,M) &=\gamma_a(i\lambda)\lim_{t\to\infty}\zeta(t) \int_{-M/t}^{M/t} \Re A(-ib) e^{-ibt}\,db\\
&+\gamma_a(i\lambda)\lim_{t\to\infty}\zeta(t) \int_{-M/t}^{M/t} \Re\left(A(1/t-ib)-A(-ib) \right)e^{-ibt}\,db\\
&=\gamma_a(i\lambda)\lim_{t\to\infty}\zeta(t)D_1(t)+\gamma_a(i\lambda)\lim_{t\to\infty}\zeta(t)D_2(t).
\end{align*}
 By equation~\eqref{eq-a2},
\(
|\zeta(t)D_2(t)|\ll t^{-(1-\eps)}\zeta(t)=o(1).
\)

It remains to estimate $\gamma_a(i\lambda)\lim_{t\to\infty}\zeta(t)D_1(t)$.
Using equation~\eqref{eq-A} we have that $A(-ib)= C_p |b|^{-1/p}\ell^*(1/|b|)^{-1}(1+o(1))$,
 where  $C_p$ is a complex constant. Hence, 
\begin{align}
\label{eq-onemore1}
\nonumber\gamma_a(i\lambda)\lim_{t\to\infty}\zeta(t)D_1(t)
\nonumber & =\gamma_a(i\lambda)\lim_{t\to\infty}\zeta(t) \int_{-M/t}^{M/t} \Re A(-ib) e^{-ibt}\,db\\
\nonumber &=\gamma_a(i\lambda)\lim_{t\to\infty}\zeta(t) \int_{-M/t}^{M/t}\Re(C_p) |b|^{-1/p}\ell^*(1/|b|)^{-1}e^{-ibt}(1+o(1))\,db\\
%\nonumber &+\gamma_a(i\lambda)\lim_{t\to\infty}\zeta(t) \int_{-M/t}^{M/t}\Re(C_p) |b|^{-1/p}\ell^*(1/|b|)^{-1}(E(b)-1)e^{-ibt}\,db\\
%&:=\gamma_a(i\lambda)\lim_{t\to\infty}\zeta(t) L_1(M,t)+ \gamma_a(i\lambda)\lim_{t\to\infty}\zeta(t) L_2(M,t).
&=2\gamma_a(i\lambda) \Re(C_p)\int_{0}^{M/t} b^{-1/p}\ell^*(1/b)^{-1}\cos tb\,db.
\end{align}
 By Lemma~\ref{lemma-derS} i), $\Re(C_p)>0$. Set $d_0:=2\Re(C_p)$. With a change of variables,
\begin{align*}
2\Re(C_p)\int_{0}^{M/t} & b^{-1/p}\ell^*(1/b)^{-1}(1+o(1))\cos tb\,db \\
&=d_0t^{-(1-1/p)}\int_{0}^{M} b^{-1/p}\ell^*(t/b)^{-1}(1+o(1))\cos b\,db.
\end{align*}
Thus,
\begin{align}
\label{eq-l1mt}
\lim_{t\to\infty}\zeta(t) D_1(M,t)
&=d_0\lim_{t\to\infty}\int_{0}^{M} b^{-1/p}\frac{\ell^*(t)}{\ell^*(t/b)}(1+o(1)) \cos b\,db \nonumber \\
&=d_0\int_{0}^{M} b^{-1/p}\cos b\,db,
\end{align}
where in the last equality we have used that $\ell^*$ is slowly varying (see, for instance,~\cite{BGT}) 
together with the dominated convergence theorem.

To conclude we just need to estimate $\int_{0}^{M} b^{-1/p}\cos b\,db$ in~\eqref{eq-l1mt}. Write
\begin{align*}
\int_{0}^{M} b^{-1/p}\cos b\,db=\int_{0}^{\infty} b^{-1/p}\cos b\,db-\int_{M}^\infty  b^{-1/p}\cos b\,db
\end{align*}
 and note that
$
|\int_{M}^\infty  b^{-1/p}\cos b\,db|\le M^{-1/p}
$.
Thus,
\begin{equation*}
\gamma_a(i\lambda)\lim_{t\to\infty}m(t) L_1(M,t)=\gamma_a(i\lambda)\frac{d_0}{\pi}\int_{0}^{\infty} b^{-1/p}\cos b\,db:=d_p,
\end{equation*}
as desired.~\end{proof}
 
%For the next result, we recall $g_a$ is supported on $[-a,a]$ and that $p>1$.
\begin{lemma} \label{lem:I2} For any $1<M$ and $M/t<a$, there exists $C, C', C''>0$ such that for any $\eps<(p-1)/2$,
\[
|\zeta(t)I_2(t,M)| \le C t^{-\frac1p (p-1)^2 + \eps}  +C't^{-1/p}\ell^*(t)^{-1}+ 2C'' t^{\eps} M^{-1/p+\eps} \ell^*(t)^{-1}.
\]
\end{lemma}

\begin{proof} Compute that
\begin{align*}
I_2(t,M) &=\frac{1}{it}\int_{-a-\lambda \leq b \leq -\lambda,\ |b|>M/t}( e^{-itb})'\gamma_a^\pm(i(b+\lambda)) A(1/t-ib)\, db.
\end{align*}
Integration by parts gives four constant terms and two integrals
\[
J_1(t, M) =\int_{-a-\lambda \leq b \leq a-\lambda,\ |b|>M/t} e^{-itb}\frac{d}{db}(\gamma_a^\pm (i(b+\lambda)) A(1/t-ib)\, db
\]
and
\[
J_2(t, M) =\int_{-a-\lambda \leq b \leq a-\lambda,\ |b|>M/t} e^{-itb}\gamma_a^\pm(i(b+\lambda))\frac{d}{db}A(1/t-ib)\, db.
\]
Of the four constant terms it suffices to look at $b = M/t$, because the other three are not larger in
absolute value. It follows from the boundedness of $\gamma_a$
and equation \eqref{eq-A} that for all $M/t \le a$ and some $C>0$,
\begin{align*}
\zeta(t)t^{-1}|(\gamma_a^\pm(i(M/t+\lambda)) A(1/t-iM/t)|&\ll m(t)t^{-1}  |A(1/t-iM/t)|\\
&\ll t^{1-1/p} \ell^*(t)^{-1} t^{-1} (t/M)^{1/p} \ell^*(t/M)^{-1}\\
&\le C'' t^\eps M^{-1/p}.
\end{align*}
Next, since $\gamma_a^\pm$ has a bounded derivative on $[-a,a]$, there is some $C'>0$ such that
\begin{align*}
\zeta(t)t^{-1}|J_1(t,M)|\ll t^{1/p} \ell^*(t)^{-1}  \int_{-a-\lambda \leq b \leq a-\lambda,\ |b|>M/t} |A(1/t-ib)|\, db
\le C't^{-1/p}\ell^*(t)^{-1}.
\end{align*}
Finally, using equation~\eqref{eq-A3}, 
\begin{align*}
|J_2(t, M)| \ll &\int_{-a-\lambda \leq b \leq a-\lambda,\ |b|>M/t} b^{-(1+1/p+\eps)}\, db \\
&+ t^{2-p+\eps}\int_{-a-\lambda \leq b \leq a-\lambda,\ |b|>M/t}  b^{-(1/p+\eps)}\, db\\
 =&\ J_2^1(t, M)+J_2^2(t, M).
\end{align*}
For the first term, compute that for any $\eps>0$,
\begin{align*}
\zeta(t)t^{-1}|J_2^1(t, M)| & \ll \zeta(t)t^{-1 +1/p+\eps} \int_{-t(a+\lambda) \leq \sigma \leq t(a-\lambda),\ |\sigma|>M} 
\sigma^{-(1+1/p-\eps)}\, d\sigma \\
&\le C'' t^{\eps} M^{-\frac1p+\eps} \ell^*(t)^{-1}.
\end{align*}
For the second term, there exist $C>0$ such that for any $\eps>0$,
\begin{align*}
\zeta(t)t^{-1}|J_2^2(t, M)|& \ll t^{2-p+\eps-1/p}\int_{-a-\lambda \leq b \leq a-\lambda,\ |b|>M/t}  b^{-(1/p-\eps)}\, db 
\le C  t^{-\frac1p (p-1)^2 + \eps},
\end{align*}
which ends the proof.~\end{proof}

\section{Asymptotics of $\hat L(u-ib,i\theta)$}
\label{subsec-sppropL}

We recall the main steps and estimates the operator $\hat L(u-ib, i\theta)$ introduced in equation~\eqref{eq-opppppp}, to be used in Section~\ref{sec-pflemmader}  below.

For $u\ge 0$, $(b,\theta)\in R\times [-\pi,\pi)$ and $v\in L^1(\mu)$, we write
\[
\hat L(u-ib, 0)v= L(e^{-(u-ib) r}v),\quad \hat L(0, i\theta)v= L(e^{i\theta \phi}v).
\]
We first consider the smoothness of $\hat L(u-ib,i\theta)$.
Under the assumption that $F$ is Gibbs Markov and $r$ satisfies (H0) and (H1), the argument of~\cite[Proposition 12.1]{MelTer17} shows that
for all $u\ge 0$,
\begin{align}
\label{eq-invtrL}
\| \frac{d}{db}\hat L(u-ib, 0)\|_\vartheta<\infty.
\end{align}
Moreover, the argument for derivatives used in~\cite[Proof of Proposition 12.1]{MelTer17} 
shows that for all $u>0$,
\(
\| \frac{d^2}{db^2}\hat L(u-ib, 0)\|_\vartheta\ll \int_Y r^2e^{-ur}\, d\mu.
\)
Here we note that the argument of~\cite[Proof of Proposition 12.1]{MelTer17} immediately applies since under (H0), $r$ is bounded below and trivially satisfies~\cite[Assumption (A1)]{MelTer17}, which is
crucially used in~\cite[Proof of Proposition 12.1]{MelTer17}.

Further, let $G(x)=\mu(r<x)$. By (H1) and Potter's bounds (see~\cite{BGT}), for all $u>0$ and for any $\eps>0$,
\begin{align*}
\int_Y r^2e^{-ur}\, d\mu &=-\int_0^\infty x^2 e^{-ux} d(1-G(x))\\
&= 2\int_0^\infty e^{-ux} x^{1-p}\ell(x)\, dx-u\int_0^\infty e^{-ux} x^{2-p}\ell(x)\, dx +O\left(\int_0^\infty e^{-ux} x^{-(\gamma-2)}\right)\\
&\ll u^{p-2-\eps}\int_0^\infty e^{-\sigma} \sigma^{1-p+\eps}\, d\sigma+u^{p-2-\eps}\int_0^\infty e^{-\sigma} \sigma^{2-p+\eps}\, d\sigma\ll u^{p-2-\eps}.
\end{align*}
Hence,
\begin{align}
\label{eq-invtrL2}
\| \frac{d^2}{db^2}\hat L(u-ib, 0)\|_\vartheta\ll u^{p-2-\eps}.
\end{align}

By~\eqref{eq-invtrL}, for all $u\ge 0$, $\hat L(u-ib)$ is continuous as a function of $b$. That is, for all $h>0$,
\begin{align*}
%\label{eq-contlb}
\|\hat L(u-i(b+h), 0) -\hat L(u-ib, 0)\|_\vartheta\ll h.
\end{align*}
By an argument similar to the one above (working with the perturbation $e^{i\theta\phi}$
instead of $e^{ibr}$ and exploiting $\phi\in L^1$) or by the argument  used in~\cite[Proof of Lemma 2.2, item 3]{Tho16},
we have that for all $h>0$,
\begin{align*}
%\label{eq-contlteta}
\|\hat L(0,i(\theta+h)) -\hat L(0,i\theta)\|_\vartheta\ll h.
\end{align*}
Putting the previous two displayed estimates together,  we have that for all $u\ge 0$ and  for all $h_1, h_2>0$,

\begin{equation}
\label{eq-contd}
\begin{array}{l}
\|\hat L(u-i(b+h_1),i(\theta+h_2)) -\hat L(u-ib,i\theta)\|_\vartheta\ll h_1+h_2,\\
\|\hat L(u-ib, i\theta) -\hat L(0, 0)\|_\vartheta\ll |u-ib|+|\theta|.
\end{array}
\end{equation}

We already know that $L$ has a simple isolated eigenvalue at $1$ (as an operator on $\B_\vartheta$). This together with the
above continuity properties for $\hat L(u-ib, i\theta)$ implies that that there exists $\delta>0$
and a continuous family of simple eigenvalues $\lambda(u-ib, i\theta)$ for $0\le u\le \delta$ and $(b,\theta)\in B_{\delta}(0,0)$ 
with $\lambda(0,0)=1$. 

Also, the arguments in~\cite[Proof of Lemma 2.6]{Tho16} carry over, ensuring that
the spectral radius of $\hat L(ib, i\theta)$ viewed as an operator on $\B_\vartheta$ is strictly less than $1$ for all $u\ge 0$ 
and all $(b,\theta)\in \R \times [\pi,\pi)\setminus\{(0,0)\}$.

\begin{rmk}
\label{rmk-proj}
With these specified we note that the estimates in~\eqref{eq-invtrL}--\eqref{eq-contd}
%,~\eqref {eq-invtrL2}
also hold for the family of eigenprojections $P(u-ib,\theta)$, $u\ge 0,b\in\R$, $\theta\in [-\pi,\pi)$ associated with the family of eigenvalues $\lambda(u-ib,i\theta)$.
\end{rmk}

\section{Proof of Lemma~\ref{lemma-derS}}
\label{sec-pflemmader}

In this section we prove Lemma~\ref{lemma-derS} via three sublemmas. 

\begin{sublemma} 
\label{subl-hatLder}
Assume (H0) and (H1). Then for all $u\ge 0$, $b\in\R$ and $\theta\in [-\pi,\pi)$, 
and for any $\eps>0$, 
\[
\|\frac{d}{db}\hat L(u-ib, i\theta)\|_\vartheta\ll 1, \quad \| \frac{d}{db}\hat L(u-ib, i\theta)-\frac{d}{db}\hat L(-ib,i\theta)\|_\vartheta \ll u^{1-\eps}
\]
and 
\(
\|\frac{d^2}{db^2}\hat L(u-ib, i\theta)\|_\vartheta\ll u^{p-2-\eps}.
\)
Moreover, the same estimates hold for the family of eigenprojections $P(u-ib,\theta)$.
%with $u\ge 0,b\in\R$ and $\theta\in [-\pi,\pi)$.
\end{sublemma}

\begin{proof} Since $e^{i\theta\phi}$ is constant on partition elements, the conclusion follows by 
the argument recalled (namely~\cite[Proposition 12.1]{MelTer17})
in obtaining~\eqref {eq-invtrL} and~\eqref {eq-invtrL2}.~\end{proof}

Recall that $\lambda(u-ib,i\theta)$ is well defined for  $0\le u\le \delta$ and $(b,\theta)\in B_{\delta}(0,0)$.
The next result gives the asymptotics of the first two derivatives
of $\lambda(u-ib,i\theta)$ in $b$; inside the proof
we also give another verification of \eqref{eq-eignddouble}.
 
\begin{sublemma}
\label{subl-double} Assume (H0) and (H1). Then  as $u,b\to 0$ and as $\theta\to 0$, 
\begin{align}
\label{eq-eignddouble}
1-\lambda(u-ib,i\theta)=(u-ib)r^* +c_p\ell_p(1/|\theta|)|\theta|^p+ o(|u-ib|+\ell(1/|\theta|)|\theta|^p),
\end{align}
where $r^* = \int_Y r\, d\mu$, $c_p$ is a positive constant and
i) if $p\in (1,2)$, $\ell_p=\ell$ with $\ell$ as in (H1);
ii) if $p=2$, $\ell_p(y)=2\int_1^y\frac{\ell(x)}{x}\, dx$.

Also, $\frac{d}{db}\lambda(u-ib,i\theta)= -ir^*(1+o(1))$.
Moreover, for all $u>0$ 
and $(b,\theta)\in B_{\delta}(0,0)$ and any $\eps>0$,
$|\frac{d^2}{db^2}\lambda(u-ib,i\theta)|\ll u^{p-2-\eps}$.
\end{sublemma}
\begin{proof} 
The asymptotic in~\eqref{eq-eignddouble} for $u>0, b=0$ is contained
in~\cite[Proof of Lemma 2.4]{Tho16}. Since we are interested in $b\ne 0$, we provide a proof below.
 
Let $v(u-ib,i\theta)$ be the eigenfunction associated with $\lambda(u-ib,i\theta)$, normalised such that $\mu(v(u-ib,i\theta))=1$.
Put $\Psi_r(u-ib)=\int_Y (1-e^{-(u-ib)r})\, d\mu$,  $\Psi_\phi(\theta)=\int_Y (1-e^{i\theta\phi})\, d\mu$ and  $ \Psi_{r,\phi}(u-ib, i\theta)=\int_Y (1-e^{-(u-ib)r})(1-e^{i\theta\phi})\, d\mu$.
Via a standard calculation (see, for instance,~\cite[Proof of Lemma 2.4]{Tho16}),
\begin{align}
\label{eq-lambdaubnonder}
1-\lambda(u-ib,i\theta) &=\Psi_r(u-ib)+\Psi_\phi (\theta)
+\Psi_{r,\phi}(u-ib, i\theta)+V(u-ib,i\theta),
\end{align}
where $V(u-ib,i\theta)=\int_Y (\hat L(u-ib, i\theta)-\hat L(0,0)) (v(u-ib, i\theta)-v(0,0))\, d\mu$. 

By (H1) and the argument used inside~\cite[Proof of Lemma 2.4]{MT13} (working with with $\beta\in (1,2]$ there)
we obtain that as $u,b\to 0$,
\begin{align}
\label{eq-lambdaubnonder11}
\Psi(u-ib)=(u-ib)r^*+\ell(u-ib)(u-ib)^p(1+o(1)).
\end{align}
Alternatively, this follows by the argument used inside~\cite[Proof of Lemma A1]{GM} 
(with $t$ there replaced by $u-ib$).

Under (H1),~\cite[Theorem 5.1]{AaronsonDenker01} ensures that for $p\in (1,2)$,
\begin{align*}
\label{eq-lambdaubnonder1}
\Psi_\phi (\theta)=c_p\ell_p(1/|\theta|)|\theta|^p (1+o(1)).
\end{align*}
If $p\in (1,2)$, $\ell_p=\ell$ with $\ell$ as in (H1) and $c_p=2\Gamma(p-1)\cos(\pi p/2)>0$, then there is no exact term  containing just $\theta$
because  $\phi$ is symmetric; in the notation  of~\cite[Theorem 5.1]{AaronsonDenker01}, the symmetry of $\phi$ gives $c_1=c_2$, $\beta=0$, $\gamma=0$,
which in turn implies the previous displayed formula.
If $p=2$,  $\ell_p=2\int_1^y\frac{\ell(x)}{x}\, dx$ with $\ell$ as in (H1), then $c_p=1/2$ by~\cite[Theorem 3.1]{ADb}.

Next, we estimate  $\Psi_{r,\phi}((u-ib),\theta)$. First, compute that for any $\eps\in (0,1)$,
\begin{align*}
|\Psi_{r,\phi}(u-ib,i\theta)| &\ll |\theta| \int_Y |\phi| |1-e^{-(u-ib)r}|\, d\mu\\
&\ll |\theta|\,\|\phi\|_{L^{2-\eps}(\mu)} \Big( \int_Y  |1-e^{-(u-ib)r}|^{(2-\eps)/(1-\eps)}\, d\mu\Big)^{(1-\eps)/(2-\eps)}\\
&\ll |\theta|\, |u-ib|^{1-\eps}\ll  |\theta|^{p+\eps}+ |u-ib|^{\frac{(1-\eps)(p+\eps)}{p+\eps-1}},
\end{align*}
where we have used Young's inequality and that $ \frac{(1-\eps)(p+\eps)}{p+\eps-1}>1$. 
Hence, $|\Psi_{r,\phi}(u-ib,i\theta)| = o(|u-ib|+|\theta|^p)= o(|\Psi_r(u-ib)+\Psi_\phi (\theta)|)$. 
Finally, by~\eqref{eq-contd}, $|V(u-ib,i\theta)|\ll (|u-ib|+|\theta|)^2$. These together with~\eqref{eq-lambdaubnonder} imply~\eqref{eq-eignddouble}.

For the second statement on the derivative, compute that for $m=1,2$,
\begin{equation}\label{eq-lambdaubder}
\frac{d^m}{db^m}\lambda(u-ib,i\theta)=-\frac{d^m}{db^m}\Psi_r((u-ib)+\frac{d^m}{db^m}\Psi_{r,\phi}(u-ib, i\theta)+O(|\frac{d^m}{db^m}V(u-ib, i\theta)|).
\end{equation}

Next, by  (H1) and the argument used  inside~\cite[Proof of Proposition 4.1]{Terhesiu16} 
(working with $\beta\in (1,2]$ there), we obtain that as $u,b\to 0$,
\[
\frac{d}{db}\Psi_r(u-ib)=-ir^*+\ell(u-ib)(u-ib)^{p-1}(1+o(1)).
\]

Recall that for all $x>0$ and $\gamma\in (0,1)$, $|1-e^{ix}|\le x^\gamma$. Note that under (H1), $|\phi|, r\in L^{p'}$ for any $1<p'<p$.
 Hence, for  $q=(1-1/p')^{-1}$, $p'<p$,
\[
\Big|\frac{d}{db}\Psi_{r,\phi}(u-ib, i\theta)\Big|\ll 
\theta^{1/q}\int_Y r |\phi|^{1/q}\, d\mu\ll \theta^{1/q}.
\] 
Thus, as $u, b,\theta\to 0$,
\[
\Big(-\frac{d}{db}\Psi_r(u-ib)+\frac{d}{db}\Psi_{r,\phi}(u-ib, i\theta)\Big) = -ir^* (1+o(1)).
\]
So far, we estimated the first two terms in the RHS of \eqref{eq-lambdaubder} (with $m=1$).
To complete the proof that $\frac{d}{db}\lambda(u-ib,i\theta)\to - ir^*$ as $u,b\to 0$,
we estimate the third term. Compute
\begin{align*}
\frac{d}{db}V(u-ib,i\theta)&=\int_Y (\frac{d}{db}\hat L(u-ib, i\theta))(v(u-ib, i\theta)-v(0,0))\, d\mu\\
&+\int_Y (\hat L(u-ib, i\theta)-\hat L(0,0)) \frac{d}{db}v(u-ib, i\theta))\, d\mu.
\end{align*}
By standard perturbation theory,  the estimates for $\hat L(u-ib,i\theta)$ carry over to the family of eigenfunctions $v(u-ib,i\theta)$.
By Sublemma~\ref{subl-hatLder} (estimates on the first derivative) and~\eqref{eq-contd}:
\begin{align*}
|\frac{d}{db}V(u-ib,i\theta)|&\ll\| \frac{d}{db}\hat L(u-ib, i\theta))\|_\infty\|\hat L(u-ib, i\theta)-\hat L(0,0)\|_\infty\\
&\ll \|\frac{d}{db}\hat L(u-ib, i\theta))\|_\vartheta\|\hat L(u-ib, i\theta)-\hat L(0,0)\|_\vartheta\ll  |u-ib|+|\theta|.
\end{align*}

We continue with the estimate on the second derivative. 
By the calculation used for deriving~\eqref {eq-invtrL2}, for $u>0$ and for any $\eps>0$,
\begin{align}
\label{eq-ppp}
|\frac{d^2}{db^2}\Psi_r(u-ib)|\ll \int_Y r^2e^{-ur}\, d\mu\ll u^{p-2-\eps}.
\end{align}
Also, $\Big|\frac{d^m}{db^m}\Psi_{r,\phi}(u-ib, i\theta)|\le  \int_Y r^m e^{-ur}|1-e^{i\theta\phi}|\, d\mu$
and similarly to~\eqref{eq-ppp},
\[
\Big|\frac{d^2}{db^2}\Psi_{r,\phi}(u-ib, i\theta)\Big|\ll u^{p-2-\eps}.
\]
Using Sublemma~\ref{subl-hatLder} 
(the estimates on the second derivatives) we compute that
\begin{align*}
|\frac{d^2}{db^2}V(u-ib,i\theta)|\ll \Big(\|\frac{d}{db}\hat L(u-ib, i\theta))\|_\vartheta\Big)^2
+\|\frac{d^2}{db^2}\hat L(u-ib, i\theta))\|_\vartheta\ll u^{p-2-\eps}.
\end{align*}
The statement on the derivatives of $\lambda$ follow by putting all the above estimates together and using~\eqref{eq-lambdaubder}.~\end{proof}

The final required estimate is

\begin{sublemma} \label{subl-mainterm} 
There exists $\eps_0>0$ so that the following hold for all $u,b\in B_{\eps_0}$.
\begin{itemize}
\item[i)] There exist positive constants $C,\tilde C$ so that $C\le |u-ib|^{1-\frac 1p}(\ell^*(1/|u-ib|))^{-1}\|S(u-ib)\|_\vartheta\le \tilde C$.
Also, there exists a complex constant $C_0$
with $\Re C_0> 0$ so that \(S(-ib)\sim i C_0 \sign(b)|b|^{\frac 1p-1}\ell^*(1/|b|)P\) as $b\to 0$.
\item[ii)] There exist a positive constant $C$ so that $\|\frac{d}{db}S(u-ib)\|_\vartheta\le  C|u-ib|^{\frac 1p-2}$. 
Also,  there exists a complex constant $C_1$
with $\Re C_1> 0$ so that
$\frac{d}{db}S(u-ib)\sim i C_1|b|^{\frac 1p-2}\ell^*(1/|b|)P$ as $b\to 0$.
\item[iii)] For any $\eps>0$,  $\|\frac{d}{db}S(u-ib)-\frac{d}{db}S(-ib)\|_\vartheta\le C_\eps u^{1-\eps}|u-ib|^{\frac 1p-2}$, for some $C_\eps>0$.
\item[iv)]For any $\eps>0$,
$\|\frac{d^2}{db^2}S(u-ib)\|_\vartheta\le C_\eps(|u-ib|^{\frac 1p-2-\eps}u^{p-2-\eps}+ |u-ib|^{\frac 1p-3-\eps})$, for some $C_\eps>0$.
\end{itemize}
%Moreover, items i)--ii) also hold for $u=0$, as $b\to 0$ (with $(u-ib)$ replaced by $-ib$).
\end{sublemma}

\begin{proof}Throughout this proof we let $Pv:=P(0,0)v=\int_Y v\,d\mu$
be the spectral projection associated with the eigenvalue $\lambda(0,0)=1$.

Although item i) follows by the argument in~\cite[Proof of Proposition 2.7]{Tho16}, we sketch the argument partly to fix the notation required for the proof of ii), partly because~\cite[Proof of Proposition 2.7]{Tho16} works with $s\in\R$ as opposed to $u-ib\in\C$ here.
As explained in Section~\ref{subsec-sppropL},  $\hat L(u-ib, i\theta): \B_\vartheta\to \B_\vartheta$ has good spectral properties. In particular, there exists $\delta>0$ such that for all $ u\in [0,\delta)$
and for all $(b,\theta)\in B_\delta(0,0)$ we can write
\begin{align}
\label{eq-decomp}
\nonumber (I-\hat L(u-ib, i\theta))^{-1}&=(1-\lambda(u-ib,i\theta))^{-1}P\\
\nonumber &+(1-\lambda(u-ib,i\theta))^{-1}(P(u-ib,i\theta)-P)\\
&+(I-\hat L(u-ib, i\theta))^{-1}Q(u-ib, i\theta),
\end{align}
where  $P(u-ib,i\theta)$ is the family of spectral projections  associated with the family
of simple eigenvalues $\lambda(u-ib,i\theta)$
and $Q=I-P$.

Since $\|(I-\hat L(u-ib, i\theta))^{-1}Q(u-ib, i\theta)\|_\vartheta\ll 1$, we have using \eqref{eq-eignddouble} and Remark~\ref{rmk-proj}, as $u, b,\theta\to 0$,
\[
 (I-\hat L(u-ib, i\theta))^{-1}=\Big((u-ib)r^*+c_p\ell_p(1/|\theta|)|\theta|^p\Big)^{-1}P(1+o(1)),
\]
where $r^* = \int_Y r\, d\mu$, $c_p$ is a positive constant and $\ell_p$ is a slowly varying function.

{\bf{Proof of i).}}
Fix $\delta$ such that~\eqref{eq-decomp} holds.
Proceeding as in~\cite[Proof of Proposition 2.7]{Tho16}, we note that

\begin{align*}
S(u-ib)&=\int_{-\delta}^{\delta}(I-\hat L(u-ib, i\theta))^{-1}\, d\theta(1+o(1)).
\end{align*}
Set $I(\theta)=c_p\, \ell_p(1/\theta)|\theta|^p$ and let 
$I^*(\theta)=\ell^*(1/|\theta|)|\theta|^{1/p}$
be the asymptotic (as $\theta \to 0$) inverse of $I$; in particular, we recall that $\ell^*$ is slowly varying.
Putting the above together,

\begin{align*}
S(u-ib)&= \int_{-\delta}^{\delta} \Big((u-ib)r^*+\, I(\theta)\Big)^{-1}P(1+o(1))\, d\theta\\
&=\frac{1}{(u-ib)r^*}\int_{-\delta}^{\delta}
\Big(1+\frac{ I(\theta)}{(u-ib)r^*}\Big)^{-1}\, d\theta P(1+o(1)).
\end{align*}
With the change of variables $\theta=\sigma I^*(|u-ib|)$,
\begin{align}
\label{eq-sned}
S(u-ib)&=\frac{I^*(|u-ib|)}{(u-ib)r^*}\int_{-\frac{\delta}{I^*(|u-ib|)}}^{\frac{\delta}{I^*(|u-ib|)}}
\Big(1+\frac{ I(I^*(|u-ib|)\sigma)}{(u-ib)r^*}\Big)^{-1}\, d\sigma P(1+o(1)).
\end{align}
Using Potter's bounds (see~\cite{BGT}) to estimate the integrand, we have for any $\delta_0>0$
\begin{align*}
|1+\frac{ I(I^*(|u-ib|)\sigma)}{(u-ib)r^*}|&=\Big|1+\frac{1}{r^*}
\frac{|u-ib|}{u-ib}\sigma^p(\ell^*(1/|u-ib|)^p\ell(1/I^*(|u-ib|)\sigma))|\Big|\\
&\ge\Big | 1+\frac{1}{r^*}\frac{|u-ib|}{u-ib}\min(|\sigma|^{p+\delta_0}, |\sigma|^{p-\delta_0})\Big|.
\end{align*}
Since $\frac{|u-ib|}{u-ib}$ has modulus $1$ for $u-ib \neq 0$, we have
\begin{align}
\label{eq-rewr}|1+\frac{I(I^*(|u-ib|)\sigma)}{(u-ib)r^*}|\ge
1-\frac{1}{r^*}\min(|\sigma|^{p+\delta_0}, |\sigma|^{p-\delta_0}).
\end{align}
Hence, the integral in~\eqref{eq-sned} is bounded and bounded away from $0$.  Also,
\begin{align}
\label{eq-rewr1}
\frac{I^*(|u-ib|)}{(u-ib)r^*} &\sim\frac{c_p^{1/p}\, |u-ib|^{1/p}}{(u-ib)r^*}\ell^*(1/|u-ib|).
%\sim\frac{ic_p^{1/p}}{r^*} |u-ib|^{\frac 1p-1}\ell^*(1/|u-ib|).
\end{align}
The first part of item i) follows.

To prove the second part of item i),  note that $\frac{|-ib|}{-ib}=i\sign(b)$.
Thus, the integrand in \eqref{eq-sned} is bounded by an absolutely integrable function
and converges pointwise to $(1\pm \frac{i}{r^*}\sigma^p)^{-1}$.
Since we also know that $\frac{\delta}{I^*(|-ib|)}\to\infty$ as $b\to 0$,
 it follows from the dominated convergence theorem that 
 \begin{align}
\label{eq-limubpaths}
\nonumber \lim_{b\to 0} &\int_{-\frac{\delta}{I^*(|-ib|)}}^{\frac{\delta}{I^*(|-ib|)}}
\Big(1\pm \frac{ I(I^*(|-ib|)\sigma)}{(-ib) r^*}\Big)^{-1}\, d\sigma\\
&= \int_{-\infty}^\infty (1\pm\frac{i}{r^*}\sign(\sigma)\, |\sigma|^p)^{-1}\, d\sigma
= \int_{-\infty}^\infty
\frac{1\mp\frac{i}{r^*}\sign(\sigma)|\sigma|^p}{|1+\frac{1}{(r^*)^2}|\sigma|^{2p}|}
=:K_p,
\end{align}
where $K_p$ is a positive constant, independent of $b$. 
% To see that $K_p>0$, just note that
% \begin{equation}
% \label{eq-realcstCp}
% \left(1+\frac{i }{r^*}\sign(\sigma)|\sigma|^p\right)^{-1}\right)=
% \frac{\left(1-\frac{i}{r^*}\sign(\sigma)|\sigma|^p\right)}{|1+\frac{c_2^2}{(r^*)^2}|\sigma|^{2p}|}>0.
% \end{equation}
Finally,  taking $u=0$ in~\eqref{eq-rewr1} we have $\frac{I^*(|-ib|)}{-ib} \sim i\sign(b)$.
The second part of item i) follows with $C_0=i\sign(b)\frac{c_p^{1/p}}{r^*} K_p>0$.

{\bf{Proof of ii).}} Differentiating~\eqref{eq-decomp} in $b$,
\begin{align}
\label{eq-rewr2}
 \frac{d}{db}(I-\hat L(u &-ib, i\theta))^{-1}=\frac{ \frac{d}{db}\lambda(u-ib,i\theta)}{(1-\lambda(u-ib, i\theta))^{2}}P
 +\frac{\frac{d}{db}\lambda(u-ib,i\theta)}{(1-\lambda(u-ib, i\theta))^{2}}(P(u-ib,i\theta)-P)\\
\nonumber&+(1-\lambda(u-ib, i\theta))^{-1}\frac{d}{db} P(u-ib, i\theta)+\frac{d}{db}(I-\hat L(u-ib, i\theta))^{-1}Q(u-ib, i\theta).
\end{align}
Using Sublemma~\ref{subl-hatLder}  (which gives the same estimates for $\frac{d}{db} P(u-ib,i\theta)$)
and~\eqref{eq-contd},
\begin{align*}
 \frac{d}{db}(I-\hat L(u-ib, i\theta))^{-1}&=
 \frac{ \frac{d}{db}\lambda(u-ib,i\theta)}{(1-\lambda(u-ib, i\theta))^{2}}P(1+o(1)).
\end{align*}
Using Sublemma~\ref{subl-double} (the estimate on the first derivative) and proceeding as in the proof of item i),
as $u,b\to 0$

\begin{align}
\label{eq-rewr3}
&\frac{d}{db}S(u-ib) \sim  -i r^*\frac{\ I^*(|u-ib|)}{((u-ib)r^*)^2}
\int_{-\frac{\delta}{I^*(|u-ib|)}}^{\frac{\delta}{I^*(|u-ib|)}}
\Big(1+\frac{ I(I^*(|u-ib|)\sigma)}{(u-ib)r^*}\Big)^{-2} d\sigma\, P.
\end{align}
By~\eqref{eq-rewr}, the integral is bounded. This together with~\eqref{eq-rewr1} gives the first part of item ii).

Next, by an argument similar to the one used in obtaining~\eqref{eq-limubpaths}, 
\[
\lim_{b\to 0}\int_{-\frac{\delta}{I^*(|-ib|)}}^{\frac{\delta}{I^*(|-ib|)}}
\Big(1+\frac{ I(I^*(|-ib|)\sigma)}{(-ib)r^*}\Big)^{-2} d\sigma=
\int_{-\infty}^\infty (1\pm\frac{i \sign(\sigma) }{r^*}\sigma^p)^{-2}\, d\sigma=:K_p',
\]
where $K'_p$ is real and positive, as we will argue below. Thus,
\[
\frac{d}{db}S(-ib)\sim \frac{-i r^*\ K_p'}{(r^*)^2}\frac{|-ib|^{1/p}}{(-ib)^{2}}P
\sim \frac{i\ K_p'}{r^*}|b|^{\frac 1p-2}P,
\]
where in the last equality we have used that $\frac{|-ib|}{-ib}=i\sign(b)$. The second part of item ii) follows
with $C_1=\frac{\ K_p'}{r^*}$.

{\bf{Showing that $K_p'$ is positive.}}
 Using the change of coordinates $r^*y = \sigma^p$ we get
 \begin{eqnarray}\label{eq:Kprime}
  K'_p &=& \int_{-\infty}^{\infty} \left( 1 - \frac{i \sign(\sigma) |\sigma|^p}{r^*}  \right)^{-2} \, d\sigma
  =
   \int_{-\infty}^{\infty} \frac{ \left( 1 + \frac{i \sign(\sigma) |\sigma|^p}{r^*}  \right)^{2} }
   {\left(1+\frac{|\sigma|^{2p}}{(r^*)^2}\right)^2}\, d\sigma \nonumber \\
& = &
   \int_{-\infty}^{\infty} \frac{ 1 - \frac{|\sigma|^{2p}}{(r^*)^2} + \frac{2 i \sign(\sigma) |\sigma|^p}{r^*}  }
   {\left(1+\frac{|\sigma|^{2p}}{(r^*)^2}\right)^2}\, d\sigma 
   = 
   2\int_{0}^{\infty} \frac{ 1 - \frac{\sigma^{2p}}{(r^*)^2}}
   {\left(1+\frac{\sigma^{2p}}{(r^*)^2}\right)^2}\, d\sigma 
    \nonumber \\
   &=& \frac{2(r^*)^{\frac1p -1}}{p}
   \int_0^\infty \frac{1-y^2}{(1+y^2)^2}\, y^{\frac1p-1} \, dt.
 \end{eqnarray}
The integrand of \eqref{eq:Kprime} is positive for $y < 0$ and negative for $y> 0$.
Hence for larger values of $p$, the factor $y^{\frac1p-1}$ puts more weight on the positive part of the integrand, and hence the integral of \eqref{eq:Kprime} is increasing in $p$.
(For $p=1$, the integral can be  computed explicitly and it is $0$.)

{\bf{Proof of iii).}} This follows by a straightforward calculation using~\eqref{eq-rewr2}, the estimate $\| \frac{d}{db}\hat L(u-ib, i\theta)-\frac{d}{db}\hat L(-ib,i\theta)\|_\vartheta\ll u^{1-\eps}$
recorded in Sublemma~\ref{subl-hatLder} and an equation similar to~\eqref{eq-rewr3}.

{\bf{Proof of iv).}} Differentiating once more in~\eqref{eq-rewr2} and using Sublemma~\ref{subl-hatLder} for the estimates for the first and second derivatives of the involved operators in $b$
together with~\eqref{eq-contd} and Sublemma~\ref{subl-double} (for both, first and second derivatives)
\begin{align*}
\| \frac{d^2}{db^2}(I-&\hat L(u -ib, i\theta))^{-1}\|_\vartheta\ll \Big|\frac{(\frac{d}{db}\lambda(u-ib,i\theta))^2}{(1-\lambda(u-ib, i\theta))^{3}}\Big|
+\Big|\frac{\frac{d^2}{db^2}\lambda(u-ib,i\theta)}{(1-\lambda(u-ib, i\theta))^{2}}\Big|\\
&\ll u^{p-2-\eps}(|u-ib|+c_p\theta^p\ell(1/|\theta|))^{-2}+(|u-ib|+c_p\theta^p\ell(1/|\theta|))^{-3}.
\end{align*}
 The conclusion follows from the previous displayed equation together with arguments  similar to the ones
used at the end of proof of item i), somewhat simplified by the fact we only study upper bounds.
~\end{proof}

We can now complete the 

\begin{pfof}{Lemma~\ref{lemma-derS}}
{\bf{Proof of i).}}
Compute that $\frac{d}{db}S(u-ib)^{-1}=-S(u-ib)^{-1}\,\frac{d}{db}S(u-ib)S(u-ib)^{-1}$. 
By the first part of Sublemma~\ref{subl-mainterm} i) (on both, upper and lower bounds) and the first part of Sublemma~\ref{subl-mainterm} ii) (on upper bounds)
we have $\|\frac{d}{db}S(u-ib)^{-1}\|_\vartheta\ll |u-ib|^{\frac 1p-1}(\ell^*(1/|u-ib|))^{-1}$.

By the second part of Sublemma~\ref{subl-mainterm} i), $S(-ib)^{-1}\sim iC_0\sign(b)|b|^{\frac 1p-1}(\ell^*(1/|b|))^{-1}P$.
By the second part of Sublemma~\ref{subl-mainterm}  ii), $\frac{d}{db}S(u-ib)\sim i C_1|b|^{\frac 1p-2}\ell^*(1/|b|)P$. Thus,
\begin{align*}
\frac{d}{db}S(-ib)^{-1}\sim i \frac{C_1}{ C_0^{2}}|b|^{-\frac1p}\ell^*(1/|b|)^{-1}P.
\end{align*}
The claimed asymptotics follows with $C_p=C_1 C_0^{-2}$.

{\bf{Proof of ii).}} This follows immediately from the formula for $\frac{d}{db}S(u-ib)^{-1}$ and Sublemma~\ref{subl-mainterm} iii).

{\bf{Proof of iii).}}
Differentiating  $\frac{d}{db}S(u-ib)^{-1}$,
\begin{align*}
\frac{d^2}{db^2}S(u-ib)^{-1}& =\Big(-S(u-ib)^{-1}\,\frac{d^2}{db^2}S(u-ib)S(u-ib)^{-1}\\
& +2\Big( S(u-ib)^{-1}\frac{d}{db}S(u-ib)\Big)^{2}S(u-ib)^{-1}\Big).
\end{align*}
The upper bounds provided by Sublemma~\ref{subl-mainterm}  i), ii) and iii)  (for $u,b$ small enough)
together with a standard calculation using further Sublemma~\ref{subl-mainterm}  ii) and iv) give the second estimate of the lemma.
\end{pfof}

\section{Krickeberg mixing in an abstract set-up}
\label{sec-abstrsetup}

Generalizing  (and correcting a mistake in the proof) a result of~\cite{doney} to operator renewal sequences, Gou\"ezel
~\cite{Gouezel11} obtains the scaling rate and thus mixing for infinite measure preserving systems
with regularly varying first return tail sequences of index $\beta\in (0,1)$.
In Subsections~\ref{sec-mainest}--\ref{sec-Mg} we translate the argument of~\cite{Gouezel11}
to the  abstract class of suspensions flows described below.

Let $(Y,\mu)$ be a probability space and assume that $(Y,F,\mu)$ is ergodic measure preserving transformation.
Let $\tau : Y\to\R_{ +}$  be a measurable nonintegrable
function bounded away from zero.
Throughout, we assume that $\essinf\tau\ge1$. Define the suspension 
$Y^\tau=\{(y,u)\in Y\times\R:0\le u\le \tau(y)\}/\sim$ where $(y, \tau(y)) \sim (Fy,0)$.
The semiflow $F_t:Y^\tau\to Y^\tau$ is defined by $F_t(y,u)=(y,u+t)$ computed modulo identifications. The measure $\mu^\tau=\mu\times Leb$ is ergodic, $F_t$-invariant and $\sigma$-finite.
Since $\tau$ is nonintegrable, $\mu^\tau(Y^\tau)=\infty$.

 Given $A,B\subset Y$, define the renewal measure
\[
U_{A,B}(I)=\sum_{n=0}^\infty \mu(y\in Y:\tau_n(y)\in I,\,y\in A,\,F^ny\in B),
\]
for any interval $I\subset\R$. We write
$U_{A,B}(x)=U_{A,B}([0,x])$ for $x>0$.

Under the assumption that
$\mu(y\in Y:\tau(y)>t)=\ell(t)t^{-\beta}$ where $\beta\in(\frac12,1]$,
~\cite[Theorem 2.3]{MT17} shows that $\lim_{t\to\infty}\ell(t) t^{1-\beta}(U_{A,B}(t+h)-U_{A,B}(t))=d_{\beta}\mu(A)\mu(B)h$
where $d_\beta=\frac{1}{\pi}\sin\beta\pi$. 
As shown in~\cite[Corollary 3.1]{MT17} (see also  Corollary~\ref{cor-mixing} below), such a result translates into
mixing for the semiflow $F_t$. The argument used in~\cite[Theorem 2.3]{MT17}
adapts and generalizes~\cite[Theorem 1]{Erickson70} to the set-up of (non iid) continuous time dynamical systems. The main steps were essentially recalled in Section~\ref{sec-stratT1},
but the definition of the measure $U$ there is different and the steps in~\cite[Proof of Theorem 1]{Erickson70} are used for a different purpose.

As clarified in~\cite{MT17}, the quantity
$U_{A,B}(t+h)-U_{A,B}(t)$ for $h>0$ can be understood in terms of twisted transfer operator for the 
map $F$ (with $\tau$ being the twist), as we explain in what follows. 
Define the symmetric measure $V_{A,B}(I)=\frac{1}{2}(U_{A,B}(I)+U_{A,B}(-I))$.
Here, $U(-I)=U(\{x: -x\in I \})$. Taking $I=[0,h]$, we get
\[
\SMALL V_{A,B}(I)=\frac{1}{2}(U_{A,B}(t+h)-U_{A,B}(t)).
\] 
Let $\H=\{\Re s>0\}$ and $\barH=\{\Re s\ge0\}$.  For $s\in\H$, define
\[
\hat R(s)v=R(e^{-s\tau}v).
\]
Under suitable spectral assumptions on the map $F$ (namely, (H)(i)-(ii) below),
\[
 \hat T(s)=(I-\hat R(s))^{-1}
\]
is well defined on $\barH\setminus\{0\}$. 
Here we clarify that the results in~\cite{Gouezel11} can be used to obtain mixing for 
 suspension flows over maps with good spectral
properties and tail for the roof function satisfying: i) $\mu(\tau>t)=\ell(t)t^{-\beta}$ where $\beta\in(0,1)$; ii) $\mu(t<\tau<t+1)=O(\ell(t)t^{-(\beta+1)})$.

To spell out the analogy between assumption (H) below  and the assumptions in~\cite{Gouezel11}, we recall briefly the terminology of operator renewal sequences
introduced in~\cite{Sarig02} to obtain lower bounds for subexponentially decaying 
(finite) measure preserving systems. Let $(X,\mu)$ be a measure space (finite or infinite), and 
$f:X\to X$ a conservative 
measure preserving map.   Fix $Y\subset X$ with $\mu(Y)\in(0,\infty)$.
Let $\varphi:Y\to\Z_{+}$ be the first return time 
$\varphi(y)=\inf\{n\ge1:f^n(y)\in Y\}$ (finite almost everywhere by conservativity).  Let $L:L^1(X)\to L^1(X)$ denote the transfer operator 
 for $f$ and 
 \begin{equation*}\label{eq:TnRn}
T_n v=1_YL^n (1_Y v),\enspace n\ge0, \qquad R_n v=1_YL^n (1_{\{\varphi=n\}}v),\enspace n\ge1.
\end{equation*}
Thus $T_n$ corresponds to general returns to $Y$ and $R_n$ corresponds to first returns to $Y$.   The relationship $T_n=\sum_{j=1}^n T_{n-j}R_j=\sum_{k=0}^\infty\sum_{j_1+j_2+\ldots+j_k=n}R_{j_1}R_{j_2}\ldots R_{j_k}$
generalizes the notion of scalar renewal sequences (see~\cite{Feller66, BGT} and references therein).
Let $\hat R(z) v=\sum_n R_n z^n$, $z\in\bar\D$. It easy to check that $\hat R(1):=R$, $R:L^1(Y)\to L^1(Y)$, is
the transfer operator associated with the induced map $F=f^\varphi$ and that $\hat R(z)v=R(z^\varphi v)$.

The mixing result~\cite[Theorem 1.1]{Gouezel11} requires that i) $\mu(\varphi>n)=\ell(n)n^{-\beta}$,  $\beta\in(0,1)$; ii) $\mu(\varphi=n)=O(\ell(n)n^{-(\beta+1)})$; iii) there exists a Banach space $\B$
with norm $\|\,\|$ such that the operator $R(z)$ has the spectral gap property and that
$\|R_n\|=O(\mu(\varphi=n))$.  Assumptions i) and ii)  are also used in~\cite{doney} to obtain a strong renewal theorem for \emph{scalar} renewal sequences with infinite mean.
 There is no direct analogue of $\|R_n\|=O(\mu(\varphi=n))$ in the setting of continuous time dynamical systems; as pointed out in~\cite{MelTer17},
in the continuous time setting, the inverse Laplace transform
of the twisted transfer operator $\hat R(s)v=R(e^{-s\tau}v)$, $s\in\H$, is just a delta function.
However, as noticed in~\cite{BMT}, $\hat R(s)$ can be related to a proper Laplace transform.
More precisely, by~\cite[Proposition 4.1]{BMT}, a general proposition on twisted transfer operators 
that holds independently of the specific properties of $F$ (see also Section~\ref{app-formulaBMT} for a very short proof), for $s\in\barH$,
\begin{align}
\label{eq-RLapl}
\hat R(s) =g_0(s)\int_0^\infty R(\omega(t-\tau)) e^{-st}\, dt=:g_0(s)\int_0^\infty M(t)\, e^{-st}\, dt,
\end{align}
where $\omega:\R\to  [0, 1]$ is an integrable function with $\supp\omega \subset [-1,1]$ and $g_0$ 
is analytic on $\H$, $C^\infty$ on any compact subset of $\{ib:b\in\R\}$
such that $g_0(0)=1$.

Recall that $\barH=\{\Re s\ge0\}$ and for $\delta, L>0$ set $\barH_{\delta, L}=(\barH\cap B_\delta(0))\cup \{ib:|b|\leq L\}$.
We assume that there exists a Banach space $\B=\B(Y)\subset L^\infty(Y)$ containing constant functions, with norm $\|\,\|_{\B}$,
 such that the following assumption holds for any $L\in (0,\infty)$ and some $\delta>0$:
\begin{itemize}
\item[\textbf{(H)}]
\begin{itemize}
\item[(i)]  The operator $\hat R : \B \to\B$ has a simple eigenvalue at $1$ and the rest of the
spectrum is contained in a disk of radius less than $1$.
\item[(ii)] The spectral radius of $\hat R(s):\B\to\B$ is less than $1$ for $s\in\barH_{\delta, L}\setminus\{0\}$.

\item[(iii)] There exists an $\omega$ satisfying \eqref{eq-RLapl}  such that
$\|M(t)\|_{\B}=O(t^{-(\beta+1)}\ell(t))$.
\end{itemize}
\end{itemize}
The assumption  $\B\subset L^\infty(Y)$ can be relaxed, it is only used for simplicity.

Assumption (H)(iii) is a natural analogue of the assumption $\|R_n\|=O(n^{-(\beta+1)})$ considered in~\cite{Gouezel11}. The present result reads as
\begin{thm}\label{thm-main_smallbeta} Assume $\mu(\tau>t)=\ell(t)t^{-\beta}$ where $\beta\in(0,1)$ with $\essinf\tau\ge1$. 
Suppose that (H) holds. Let $A,B\subset Y$ be measurable and suppose that
$1_A\in\B$.  Then for any $h>0$, 
\[
\lim_{t\to\infty}\ell(t) t^{1-\beta}(U_{A,B}(t+h)-U_{A,B}(t))=d_{\beta}\mu(A)\mu(B)h,
\]
where $d_\beta=\frac{1}{\pi}\sin\beta\pi$.
\end{thm}

%As in~\cite{MT17}, Theorem~\ref{thm-main_smallbeta} leads to mixing:

\begin{cor}{~\cite[Corollary 1]{MT17}}
\label{cor-mixing}
Assume the conclusion of Theorem~\ref{thm-main_smallbeta}. Let $A_1=A\times[a_1,a_2]$, $B_1=B\times[b_1,b_2]$ be measurable subsets of $\{(y,u)\in Y\times\R:0\le u\le \tau(y)\}$ 
(so $0\le a_1<a_2\le\essinf_A\tau$, $0\le b_1<b_2\le\essinf_B\tau$). Suppose that $1_A\in\B$. Then
$\lim_{t\to\infty} \ell(t) t^{1-\beta}\mu^\tau(A_1\cap F_t^{-1}B_1)  =d_\beta\mu^\tau(A_1)\mu^\tau(B_1)$.
\end{cor}
The proof of {Corollary~\ref{cor-mixing} goes word for word as~\cite[Proof of Corollary 3.1]{MT17} with Theorem~\ref{thm-main_smallbeta} replacing~\cite[Theorem 2.3]{MT17}.

\subsection{Main estimates and proof of Theorem~\ref{thm-main_smallbeta}}
\label{sec-mainest}

As shown in~\cite[Proposition 2.1]{MT17}, under (H) (in fact, a much weaker form of (H)(iii) here is required there),
the following inversion formula for the measure $V_{A,B}$ (a generalization of~\cite[Inversion formula, Section 4]{Erickson70} to the non iid setting) holds all $\lambda,t\in\R$,
\begin{align}
\label{eq-inversion}
\int_{-\infty}^\infty e^{-i\lambda(x-t)}\hat g(x-t)\,dV_{A,B}(x)=\int_{-\infty}^\infty e^{-itb} g(b+\lambda)\Re  \int_B \hat T(ib)1_A\,d\mu\, db,
\end{align}
where $g:\R\to\R$ is a continuous compactly supported function
with Fourier transform
$\hat g(x)=\int_{-\infty}^\infty e^{ixb} g(b)\, db$
satisfying $|\hat g(x)|=O(x^{-2})$ as $x\to\infty$.

Under (H),  $\hat T(s)=(I-\hat R(s))^{-1}$ is well defined for all $s\in\barH_{\delta, L}$, $\delta>0$, $L\in (0,\infty)$.
Continuing from~\eqref{eq-inversion}  we write
\begin{align}
\label{eq-splitting}
\nonumber\int_{-\infty}^\infty & e^{-i\lambda(x-t)}\hat g(x-t)\,dV_{A,B}(x)=\int_{-\infty}^\infty e^{-itb} g(b+\lambda)\sum_{k:t<Ka_k}\Re  \int_B \hat R(ib)^k1_A\,d\mu\, db\\
&+\int_{-\infty}^\infty e^{-itb} g(b+\lambda)\sum_{k:t\geq Ka_k}\Re  \int_B \hat R(ib)^k1_A\,d\mu\, db=:u_1(t)+u_2(t),
\end{align}
where the sequence $a_k$ is such that $\tau_k/a_k$ satisfies the local limit theorem
and $K \ge 1$ is some fixed number to be specified at the end of the present section.
Under the assumptions of Theorem~\ref{thm-main_smallbeta} (for the map $F$ and observable
$\tau$), such a local limit theorem
is known to hold,
with $a_k$ such that $a_k^\beta=k\ell(a_k)(1+o(1))$ (see~\cite{AaronsonDenker01}). The splitting in the sum above follows the analogue
pattern in the discrete time scenario outlined in~\cite{doney, Gouezel11}. In fact, the computation for the term $u_1(t)$ defined in~\eqref{eq-splitting}
goes word for word (apart from obvious differences in notation) as in~\cite[Proof of Proposition 1.5]{Gouezel11} (see also~\cite[Remark 2.1]{Gouezel11}).
Defining $A(x)=x^{\beta}/\ell(x)$ such that $A(k)=k(1+o(1))$ we write
\begin{align*}
u_1(t)&=\int_{-\infty}^\infty e^{-itb} g(b+\lambda)\sum_{k:k>A(t/K)}\Re  \int_B \hat R(ib)^k1_A\,d\mu\, db\\
&=\int_{-\infty}^\infty e^{-itb} g(b+\lambda)\Re  \int_B \hat R(ib)^{A(t/K)}\hat T(ib)1_A\,d\mu\, db.
\end{align*}
Arguing as~\cite[Proof of Theorem 1]{MT17}(see also~\cite[Remark 2.1]{Gouezel11}),
for any $K\geq 1$,
\begin{align*}
\lim_{t\to\infty}\ell(t) t^{1-\beta}\int_{-1/t}^{1/t} e^{-itb} g(b+\lambda)\Re  \int_B \hat R(ib)^{A(t/K)}\hat T(ib)1_A\,d\mu\, db
=d_{\beta}\mu(A)\mu(B).
\end{align*}
Under (H)(i)--(iii), $\|\hat R(ib)^{A(t/K)}\|_{\B}$
decays exponentially fast for $b$ outside a neighborhood of $0$ (see, for instance,~\cite[Proof of Proposition 1.5]{Gouezel11}
and~\cite{AaronsonDenker01}), which enables us to conclude that
\begin{align}
\label{eq-ut1}
\lim_{t\to\infty}\ell(t) t^{1-\beta} u_1(t)=d_{\beta}\mu(A)\mu(B).
\end{align}
It remains to estimate the term $u_2(t)$ defined in~\eqref{eq-splitting}. In~\cite{doney, Gouezel11}, the estimate for the analogue 
of this term in the discrete time setting is the hard part of their argument. Here, we  translate their argument to the notation of the present setting.

As already mentioned, in the discrete time scenario the renewal sequence $T_n$
can be written as $T_n=\sum_{k=0}^\infty\sum_{j_1+j_2+\ldots+j_k=n}R_{j_1}R_{j_2}\ldots R_{j_k}$.
An analogue of this formula in the continuous time setting can be obtained
from~\eqref{eq-inversion} using  (H)(iii).
Here we write $\hat M(ib) = \int_0^\infty M(t) e^{ibt}\ dt$ and vectors $\multi = (t_1, \dots, t_k)$
to abbreviate multiple integrals.

\begin{align*}
&\int_{-\infty}^\infty  e^{-i\lambda(x-t)}\gamma(x-t)\,dV_{A,B}(x)=
\int_{-\infty}^\infty e^{-itb} g(b+\lambda)\sum_{k\geq 0}\Re \int_B \hat R(ib)^k1_A\,d\mu\, db\\
&=\int_{-\infty}^\infty e^{-itb} g(b+\lambda)\Re\Big( \sum_{k\geq 0}g_0(ib)^k \int_B \hat M(ib)^k1_A\,d\mu\Big)\, db\\
&=\int_{-\infty}^\infty e^{-itb} g(b+\lambda)\\
&\times \Re
\Big( \sum_{k\geq 0}g_0(ib)^k \Big(\int_0^\infty \int_B \Big(\int_{t_1+\ldots+t_k=t} M(t_1)\ldots M(t_k)\, d\multi\Big)1_A\,d\mu\Big)e^{ibt}\, dt\Big)\, db.
\end{align*}
Hence, we can write
\begin{align*}
u_2(t)
&=\int_{-\infty}^\infty e^{-itb} g(b+\lambda)\\
&\times \Re
\Big( \sum_{k:t\geq Ka_k}g_0(ib)^k \Big(\int_0^\infty\int_B \Big(\int_{t_1+\ldots+t_k=t} M(t_1)\ldots M(t_k)\, d\multi\Big)1_A\,d\mu\Big)e^{ibt}\, dt\Big)\, db.
\end{align*}

The results below gives the main estimate for handling $u_2(t)$; 
the proof is deferred to Subsection~\ref{subsec-analG}.
\begin{prop}
\label{prop-analGoue}
For $t\geq a_k$, define
\begin{align*}
u_2(t, k)
&=\int_{-\infty}^\infty e^{-itb} g(b+\lambda)\\
&\times \Re
\Big(g_0(ib)^k \int_0^\infty  \Big( \int_B \Big(\int_{t_1+\ldots+t_k=t} M(t_1)\ldots M(t_k)\, d\multi\Big)1_A\,d\mu\Big)e^{ibt}\, dt\Big)\, db.
\end{align*}
Then for every $t\geq a_k$, $|u_2(t,k)|\ll k t^{-(1+\beta)}\ell(t)$.
\end{prop}

It follows from Proposition~\ref{prop-analGoue} that for any $\delta>0$,
\begin{align*}
|u_2(t)|\ll t^{-(1+\beta)}\ell(t)\sum_{k:t\geq Ka_k} k\ll t^{-(1+\beta)}\ell(t) A(t/K)^2&\ll t^{-(1-\beta)}\ell(t)K^{-2\beta}\frac{\ell(t)}{\ell(t/K)}\\
&\ll t^{-(1-\beta)}\ell(t)K^{-(2\beta-\delta)},
\end{align*}
where the last estimate was obtained using Potter's bounds (see, for instance,~\cite{BGT}). Since $K^{-(2\beta-\delta)}=o(1)$ as $K\to\infty$, we obtain
\(
|u_2(t)|=o( t^{-(1-\beta)}\ell(t)),
\)
which together with~\eqref{eq-ut1} concludes the proof of Theorem~\ref{thm-main_smallbeta}.

\subsection{Proof of Proposition~\ref{prop-analGoue}}
\label{subsec-analG}

Translating the strategy and estimates in~\cite{Gouezel11}, in what follows we consider separately
the contributions of different $(t_1\ldots t_k)$  to $u_2(t,k)$ depending on the size the indices $t_1\ldots t_k$, when compared to a
  truncation level $t_\eta$ defined
as follows. Write $t=w a_k$ for some $w\geq 1$
and let $t_\eta=w^\gamma a_k/2\in[a_k/2, t/2]$ for some $\gamma\in (0,1)$ (to be specified below). 
Let $T=\{(t_1,\ldots, t_k):t_1+\ldots+t_k=t\}$ be a set which is partitioned into four disjoint sets $T_j, j\in\{0,1,2,3\}$ as follows
\begin{eqnarray*}
 T_3&=&\{\multi \in T:\exists p, t_p\geq t/2\}\\
 T_2&=&\{\multi\in T:\forall p, t_p< t/2\mbox{ and }\exists u<v\mbox{ such that }t_u,t_v\geq t_\eta\}\\
T_1&=&\{\multi\in T:\forall p, t_p< t/2\mbox{ and }\exists !\, u \mbox{ such that }t_u\geq t_\eta\}\\
T_0&=&\{\multi\in T:\forall p, t_p< t_\eta\}.
\end{eqnarray*}

Recall (from text after~\eqref{eq-inversion})  that $g:\R\to\R$ is a continuous compactly supported function
and let $[-a,a]=\supp g$. Let $\chi:\R\to [0,1]$ be a $C^\infty$ function
supported in $[-a-3, a+3]$
such that $\chi\equiv 1$  on $[-a-2, a+2]$. 

Under (H)(iii), let $g_0(ib)$ be as defined in~\eqref{eq-RLapl} and set
\begin{align}
\label{eq-mg}
m_{g}(ib)=\begin{cases}\chi(b)g_0(ib), & b\in [-a-3, a+3]\\ 0, &\mbox{otherwise.}
\end{cases}
\end{align}
Because  $m_g(ib)$ is $C^\infty$ (since $g_0(ib)$ is $C^\infty$ on any compact interval), a quick computation using integration by parts shows  the inverse
Laplace transform of $m_g(ib)$, which we denote by $m_g(t)$, satisfies $|m_g(t)|=O(t^{-2})$. Moreover, by the same argument, for any $k\geq 1$, the inverse Fourier transform $m_g(t, k)$ of  $m_g(ib)^k$
is $O(t^{-2})$.

 Using~\eqref{eq-mg}, define
\begin{align}
\label{eq-tilde}
\hat M_g(ib)=\begin{cases} m_g(ib)\hat M(ib), & b\in\supp g,\\ 0, &\mbox{otherwise.}
\end{cases}
\end{align}
%Since $|m_g(t)|=O(t^{-2})$ and (H)(iii) holds, writing $\hat M_g(ib)=\int_0^\infty M_g(t) e^{ibt} dt$, we have
%\begin{align}
%\label{eq-tilde2}
%\|M_g(t)\|_{\B}=O(t^{-(\beta+1)}).
%\end{align}

The proof of the result below is deferred to Subsection~\ref{sec-proofpropint} and it allows us
to complete the  proof of Proposition~\ref{prop-analGoue}.
\begin{prop}
\label{prop-int}
For any $t\geq a_k$ and every $j\in\{0,1,2,3\}$, the integrals
\begin{align*}
I_j(t)=\int_{t_1+\ldots+t_k=t;\, \multi \in T_j}M_g(t_1)\ldots M_g(t_k)\, d\multi
\end{align*}
satisfy $\|I_j(t)\|_{\B}\ll k t^{-(1+\beta)}\ell(t)$.
\end{prop}
We can now complete

\begin{pfof}{Proposition~\ref{prop-analGoue}}
Note that $k\geq 1$, $u_2(t, k)$ defined in the statement of Proposition~\ref{prop-analGoue} can be written as
\begin{align*}
 u_2(t, k)
&=\int_{-\infty}^\infty e^{-itb} g(b+\lambda)\\
&\times\Re
\Big(\int_0^\infty \Big(\int_B \Big(\int_{t_1+\ldots+t_k=t} M_g(t_1)\ldots M_g(t_k)\, d\multi\Big)
1_A\,d\mu\Big)e^{ibt}\, dt\Big)\, db.
\end{align*}

By Proposition~\ref{prop-int}, for every $j\in \{0, 1, 2, 3\}$ and all $t\geq a_k$, we have 
$\|I_j(t) \|_\B=\|\int_{t_1+\ldots+t_k=t} M_g(t_1)\ldots M_g(t_k)\, d\multi\|_\B=O(k t^{-(1+\beta)}\ell(t))$. 
Since $\B\subset L^\infty(Y)$, the inverse Fourier transform of 
$\int_0^\infty \Big( \int_B  \Big(\int_{t_1+\ldots+t_k=t} M_g(t_1)\ldots M_g(t_k)\, d\multi\Big)1_A\,d\mu\Big)e^{ibt}\, dt$
is $O(k t^{-(1+\beta)}\ell(t))$.

 Recall (from text after~\eqref{eq-inversion})  that 
$\hat g(t)=\int_{-\infty}^\infty e^{itb} g(b)\,db$
satisfies $\hat g(t)=O(t^{-2})$. Taking a convolution, we obtain that for all $t\geq a_k$,
the inverse Fourier transform of $g(b+\lambda)
\Big( \int_B \Big(\int_{t_1+\ldots+t_k=t} M_g(t_1)\ldots M_g(t_k)\, d\multi\Big)1_A\,d\mu\Big)$
 is $O(k t^{-(1+\beta)}\ell(t))$. Thus, for every $t\geq a_k$, $| u_2(t, k)|=O(k t^{-(1+\beta)}\ell(t))$, as required.~\end{pfof}

\subsection{Proof of Proposition~\ref{prop-int}}
\label{sec-proofpropint}

 In this section we state two lemmas, which are the key estimates required in the proof of  Proposition~\ref{prop-int}
and are the direct analogues of~\cite[Lemmas 3.1 and 3.2]{Gouezel11}.
Throughout, $\hat M_g^{(z)}(s)=\int_0^{z} M_g(t) e^{-st} dt$ will denote a truncated version of the  Laplace transform $\hat M_g(s)$
with truncation level $z$.

Let $G:\R\to\B$ be an operator-valued function, where $\B$ is a Banach space with norm $\|\, \|_\B$. In what follows, we let $\mathcal{\hat R}$ be the non-commutative Banach algebra of continuous functions
$G:\R\to\B$ such that their Fourier transform $\hat G:\R\to\B$ lies in $L^1(\R)$, with norm $\|G\|_{\mathcal{\hat R}}=\int_{-\infty}^\infty\|\hat G(\xi)\|_{\B}\,d\xi$.
Using this, we further
let $\mathcal{\hat R}_{\beta+1}=\{G\in\mathcal{\hat R}:\sup_{\xi\in\R}\ell(|\xi|)|\xi|^{\beta+1}\|\hat G(\xi)\|_{\B}<\infty\}$ be the non-commutative Banach algebra of continuous functions with norm
$\|G\|_{{\mathcal{\hat R}_{\beta+1}}}=\int_{-\infty}^\infty\|\hat G(\xi)\|_{\B}\,d\xi+\sup_{\xi\in\R}\ell(|\xi|)|\xi|^{\beta+1}\|\hat G(\xi)\|_{\B}$.

Lemma~\ref{lemma-invlaplMpower} below guarantees that the Fourier transform $\hat M_g^{(z)}(ib)^k$, for $k\geq 1$ and $z$ large enough,
lies in the Banach algebra $\mathcal{\hat R}_{\beta+1}$; this is an analogue of \cite[Lemma 3.1]{Gouezel11}, which is the hardest estimate in the overall argument.
The proof of Lemma~\ref{lemma-algMpower} is provided in Section~\ref{sec-Mg}.

\begin{lemma}\label{lemma-algMpower}
There exists a constant $C>0$ such that $\|\hat M_g^{(z)}(ib)^k\|_{\mathcal{\hat R}_{\beta+1}}\leq C$, for all $k\geq 1$ and  $z\in [a_k/2,\infty]$.~\end{lemma}

The result below provides an estimate for the inverse Laplace transform $M_g^{(z)}(t)^k$ of  $\hat M_g^{(z)}(s)^k$, $s\in\H$ for $k\geq 1$ and $z$ large enough.

\begin{lemma}
\label{lemma-invlaplMpower}
There exists a constant $C>0$ such that for all $k\geq 1$, $z\in [a_k/2,\infty]$
and $t>0$,
\[
\| M_g^{(z)}(t)^k\|_{\B}\leq C e^{-t/z} a_k^{-1}.
\]~\end{lemma}
\begin{proof} Starting from assumption (H) and using the continuity  Lemma~\ref{lem-continf} below, the conclusion follows arguing
 word for word as in~\cite[Proof of Lemma 3.2]{Gouezel11}.~\end{proof}

\begin{pfof}{Proposition~\ref{prop-int} }The arguments for estimating $I_j(t)$, $j\in\{0,1,2,3\}$ go word for word as
the arguments used in~\cite{Gouezel11} in estimating $\sum_j$, $j\in\{0,1,2,3\}$ there with
Lemma~\ref{lemma-algMpower} replacing~\cite[Lemmas 3.1]{Gouezel11}
and Lemma~\ref{lemma-invlaplMpower} replacing~\cite[Lemma 3.2]{Gouezel11}.~\end{pfof}

\subsection{Proof of Lemma~\ref{lemma-algMpower}}
\label{sec-Mg}

Based on (H)(iii) we have the following continuity property for $\hat R$:

\begin{lemma}\label{lem-continf} 
There exists $C>0$, such that
for all $s_1,s_2\in\barH\cap \{ib:|b|\le L\}$ with $L<\infty$,
\[
\|\hat R(s_1)-\hat R(s_2)\|_{\B}\le C\, |s_1-s_2|^\beta\ell(|s_1-s_2|).
\]
\end{lemma}

\begin{proof}
By (H)(iii), $\hat R(s)=g_0(s)\hat M(s)$ where $\hat M(s)=\int_0^\infty M(t) e^{-st} dt$
with $\|M(t)\|_{\B}=O(t^{-(\beta+1)})$. Let $N=|s_1-s_2|\ell(|s_1-s_2|)$. Clearly , for all $s_1,s_2\in\barH$,
\begin{align*}
\|\hat M(s_1)  -\hat M(s_2)\|_\B &\leq |s_1-s_2|\int_0^N t\|M(t)\|_\B\, dt+2 \int_N^\infty \|M(t)\|_\B\, dt\\
& \leq |s_1-s_2| N^{1-\beta}
+2N^{-\beta}\leq C |s_1-s_2|^\beta\ell(|s_1-s_2|),
\end{align*}
for some $C>0$. Now restrict to $s\in\barH$ with $|s|\le L$.
By  equation~\eqref{eq-RLapl},  $|g_0(s)|\ll 1$ and $|g_0(s_1)-g_0(s_2)|\ll |s_1-s_2|$.
The result follows.
\end{proof}

By Lemma~\ref{lem-continf}, the map $s\mapsto \hat R(s)$ is continuous.
By (H), $\hat R(0)$ has $1$ as a simple eigenvalue, so
there exists $\delta>0$ and a continuous family $\lambda(s)$ of simple
eigenvalues of $\hat R(s)$ for $s\in\barH\cap B_\delta(0)\setminus\{0\}$ with
$\lambda(0)=1$.  Let $P(s)$ denote the
corresponding family of spectral projections, given by
\begin{align}
\label{eq-P}
P(s)=\int_{|\xi-1|=\delta}(\xi-\hat R(s))^{-1}\, d\xi.
\end{align}
 For $s\in\barH\cap B_\delta(0)\setminus\{0\}$,
 write $\hat R(s)=\lambda(s)P(s)+Q(s)$,
where $Q(s)=I-P(s)$. Recall that $\hat R(s)=g_0(s)\hat M(s)$, where $g_0$ is a scalar function. 
Hence, for $k\geq 1$,
\[
\hat M(s)^k=g_0(s)^{-k}\lambda(s)^kP(s)+g_0(s)^{-k}Q(s)^k.
\]
Recalling the definition of $\hat M_g(ib)$ in~\eqref{eq-tilde} and restricting to $b\in (-\delta, \delta)$, we get
\begin{align}
\label{eq-Mg2}
\hat M_g(ib)^k=\lambda(ib)^k m_g(ib)^{k}P(ib)+m_g(ib)^{k}Q(ib)^k.
\end{align}
 
Lemma~\ref{lemma-algMpowernontr} below is a version of Lemma~\ref{lemma-algMpower} for the non-truncated Fourier transform; this is the analogue of~\cite[Lemma 4.2]{Gouezel11}.
 Given Lemma~\ref{lemma-algMpowernontr} below, the proof of Lemma~\ref{lemma-algMpower} for estimating  the truncated Fourier transform follows goes word for word as in~\cite[Proof of Lemmas 3.1]{Gouezel11}.
\begin{lemma}
\label{lemma-algMpowernontr}
There exists a constant $C>0$ such that for all $k\geq 1$,
\[
\|\hat M_g(ib)^k\|_{\mathcal{\hat R}_{\beta+1}}\leq C.
\]
~\end{lemma}
\begin{proof}
We first assume that $\lambda(ib)$ is defined for $b\in\R$, vanishing outside the 
support of the function $g$, namely outside $[-a,a]$, $a>0$. Under this assumption, $P(ib), Q(ib)$ are also defined for $b\in\R$, vanishing outside
 outside $[-a,a]$.
 This is an analogue of the initial assumption in~\cite[Proof of Lemma 4.2]{Gouezel11}
 that the eigenvalue $\lambda(ib)$
is well defined on the whole unit circle. The general case can be dealt with as in~\cite[Proof of Lemma 4.2]{Gouezel11}, by constructing a function $\tilde R(ib)$ that coincides with $\hat R(ib)$
in a neighborhood of $0$ and it is close to $\hat R(0)$, elsewhere. The existence of such $\tilde R$ is ensured by Proposition~\ref{prop-tildeR} below.

Assuming that $\lambda(ib)$ is well defined on $[-a,a]$, we clarify that each quantity appearing in~\eqref{eq-Mg2}
lies in the  Banach algebra $\mathcal{\hat R}_{\beta+1}$.

From the text below~\eqref{eq-mg}, we know that the inverse Fourier transform of $m_g(ib)$
is $O(t^{-2})$. Next, by~\eqref{eq-P}, assumption (H)(iii) and Wiener's Lemma~\ref{lem-W}, we obtain $P(ib)\in \mathcal{\hat R}_{\beta+1}$.
%Also, by~\eqref{eq-tilde2}, the inverse  inverse Fourier transform of $\hat M_g(ib)$ satisfies $\|M_g(t)\|_\B=O(t^{-\beta+1})$.
Also, recall that $Q(ib)$ is an operator acting on $\B$ well defined on  $[-a,a]$ with
spectrum contained in a ball of radius strictly less than $1$. Thus, the spectrum
of $Q(ib)^k$ is contained in a ball of radius strictly less than $\rho^k$, for some $\rho<1$.
Hence, $Q(ib)\in \mathcal{\hat R}_{\beta+1}$.

It remains to clarify that $\lambda\in\mathcal{R}_{\beta+1}$.  The lack of the hat in $\mathcal{R}_{\beta+1}$ means that we look at a commutative Banach algebra (similar to $\mathcal{\hat R}_{\beta+1}$; see Appendix~\ref {sec-W} for precise definition),
since $\lambda(ib)$ is a scalar.
Under the extra assumption that the operator $\hat R$,  and  thus $\lambda$, is a  $2\pi$-periodic continuous function
supported on $(-\pi,\pi]$, this follows as in~\cite[Proof of Lemma 4.2]{Gouezel11} with the algebra $\mathcal{R}_{\beta+1}$ replaced by $\A_{\beta+1}$ recalled in Appendix~\ref{sec-W}).

To reduce to the situation of~\cite[Lemma 4.2]{Gouezel11}
 let $R^*$ denote the $2\pi$ periodic version of $\hat R$ and let $\lambda^*$ be its corresponding eigenvalue. Note that $\lambda|_{[-\pi,\pi]}=\lambda^*$.
As in~\cite[Proof of Lemma 4.2]{Gouezel11}, $\lambda^*\in\A_{\beta+1}$
and for any $k\geq 1$,
$|(\lambda^*)^k|_{\A_{\beta+1}}\leq C$, for  some $C>0$
(independent of $k$). Since we also know that $(\lambda^*)^k=\lambda^k|_{[-\pi,\pi]}$, a version of Wiener's Lemma for functions with compact support, namely Lemma~\ref{lem-AR} below,
ensures that $|\lambda(ib)^k|_{\R_{\beta+1}}\leq C$, for some $C>0$, as required.~\end{proof}

\section{ Verifying (H) for the flow $(\Psi_t)_{t\in\R}$  and proof of Theorem~\ref{thm-mr}}
\label{sec-verH}

First, it is easy to see that assumptions (H0)(i)--(ii) on $(r,\phi)$ imply (H)(i)-(ii) for the twisted transfer operator $R(e^{-s\tau})$, $s\in\barH$.
In particular, the joint aperiodicity of $(r,\phi)$  implies that $\tau$ is aperiodic, checking (H)(ii).

\subsection{Verification of (H)(iii) via Theorem~\ref{th-tailT}}

Assumption (H)(iii) is verified by Proposition~\ref{prop-Hiii}  below and Theorem~\ref{th-tailT}.
Proposition~\ref{prop-Hiii}  follows by the argument used in~\cite[Proposition 6.3]{BMT} (phrased under much weaker assumptions on the roof function of suspension flows).
I thank Ian Melbourne for the choice of $\omega$ below, the  key ingredient in the proof of Proposition~\ref{prop-Hiii}  below, and for allowing me to use it.

Recall  from Section~\ref{sec-stratT1} that $(Y,\tilde F,\tilde\alpha, \mu)$ is Gibbs Markov. Also recall from Remark~\ref{rmk-sphatr} that the perturbed transfer $R(e^{-s\tau})$, $s\in\barH$ associated with $\tilde F$ and twist $\tau$ has good spectral properties in $\B_{\vartheta_0}$
with norm $\|\, \|_{\vartheta_0}$.
 Recall that as in equation~\eqref{eq-RLapl}, $\omega : [−1, 1]\to [0, 1]$ satisfies $\int_{-1}^1\omega(t) dt = 1$. We choose 
\[
\omega(t-x)=\begin{cases} 
1-(x-t), & t-1 < x \le t, \\
1+(x-t), & t \le x < t+1, \\
0, & \text{otherwise.} 
\end{cases}
\]
Note that $\omega$ is uniformly Lipschitz, with Lipschitz constant $1$.

\begin{prop}
\label{prop-Hiii}Assumption (H)(iii) holds with
$\B = \B_\vartheta$, namely $\|R(\omega(t-\tau)\|_\vartheta\le C \mu(t-1<\tau<t+1)$.
\end{prop}

\begin{proof}
 By (H0), $r$ is Lipschitz and $F$ is Gibbs Markov and in particular
uniformly expanding.
 Therefore $\tau$ is Lipschitz as well, say $|\tau(y) - \tau(y')| \leq C_L
d_\vartheta(y,y')$ for
 all $a \in \tilde \alpha$ and $y,y' \in a$.
As a consequence, $y \mapsto \omega(t - \tau(y))$ is also Lipschitz with
Lipschitz constant $C_L$ and
clearly $\omega(t-\tau) \in [0,1]$  is supported on $\{ t-1 \leq \tau
\leq t+1\}$.

Since $\tilde F$ is Gibbs Markov as well, there are constants $C_1, C_2 > 0$ such that
the Jacobian $e^{\tilde p(y)}$ satisfies
$e^{\tilde p(y)} \leq C_1\mu(a)$ and $|e^{\tilde p(y)} - e^{\tilde
p(y')}| \leq C_2\mu(a)$
for all $a \in \tilde \alpha$ and $y,y' \in a$. Thus,
$$
\| R(\omega(t-\tau))v \|_\vartheta \leq
\sum_{\stackrel{a \in \tilde \alpha}{a \cap \{ t-1 \leq \tau \leq t+1\}
\neq \emptyset}}
\hspace{-5mm}
 \left(C_2 |v|_\infty + C_1 C_L |v|_\infty + C_1 \|v\|_\vartheta +
C_1 | v |_\infty \right) \mu(a).
$$
Because $\tau$ is Lipschitz (whence $\sup_a \tau - \inf_a \tau \leq C_L$),
$a \cap \{ t-1 \leq \tau \leq t+1\} \neq \emptyset$ implies
that $a \subset \{ t-1-C_L \leq \tau \leq t+1+C_L\}$.
Therefore $\| R(\omega(t-\tau))v \|_\vartheta \ll \mu(\{ t-1-C_L \leq \tau
\leq t+1+C_L\}) \| v \|_\vartheta$ as required.
\end{proof}

With (H) verified, Theorem~\ref{thm-mr} follows from  Theorem~\ref{thm-main_smallbeta}.

\appendix

\renewcommand{\thesubsection}{\Alph{section}.\arabic{subsection}}

\section{Some previous established results used in Section~\ref{sec-abstrsetup}}
\label{sec-app}

\subsection{Proof of Equation~\eqref{eq-RLapl}}
\label{app-formulaBMT}

We quickly verify~\eqref{eq-RLapl} (based on~\cite{BMT}).
Let $\omega$ be an integrable function supported on $[-1, 1]$ 
such that $\int_{-1}^1 \omega(t)\, dt = 1$ and for $s\in\bar H$, set $\hat\omega(s)=\int_{-1}^1 e^{-st} \omega(t)\, dt$. Note that $\hat\omega(s)$ is analytic on $\H$, $C^\infty$ on  any compact interval of $\{ib:b\in\R\}$
and $\hat\omega(0)=1$. Since $\tau\ge 1$ and $\supp\omega\subset [-1,1]$,
$
\int_0^\infty \omega(t-\tau) e^{-st}\, dt=e^{-s\tau}\int_{-\tau}^\infty \omega(t)\, e^{-st}\, dt=e^{-s\tau}\hat\omega(s)
$.
Hence,
\[
\int_0^\infty R(\omega(t-\tau)v) e^{-st}\, dt=R(\int_0^\infty \omega(t-\tau) v e^{-st}\, dt)=\hat\omega(s)\hat R(s) v.
\]
Formula~\eqref{eq-RLapl} follows with $g_0(s)=1/\hat\omega(s)$, so $g_0(0) = 1$, $g_0$ is analytic on $\H$ and $C^\infty$ on  any compact of $\{ib:b\in\R\}$.

\subsection{A result used in the proof of Lemma~\ref{lemma-algMpowernontr}}
The result below was established in~\cite{MelTer17} and it holds in the present setting due to Lemma~\ref{lem-continf}.
Although, \cite[Proposition 13.4]{MelTer17} is stated and proved using $\B=\B_\vartheta$,  the proof goes word for  word the same, with a general Banach space $\B$ provided that (H)(i)-(iii)
and Lemma~\ref{lem-continf} hold. 

\begin{prop}{~\cite[Proposition 13.4]{MelTer17}  }
\label{prop-tildeR} Assume (H)(i)-(iii) and recall $\beta\in(0,1)$.
Let $p<\beta$,  let $\epsilon>0$ and
let $\delta>0$.  For all $r>0$ sufficiently small, there exists 
a $C^{p-\epsilon}$ family $b\mapsto \tilde R(b)$ with a $C^{p-\epsilon}$ family of simple eigenvalues
$\tilde\lambda(b)\in\{s\in\C:|s-1|<\delta\}$ such that
\begin{itemize}
\item[(a)] $\tilde R(b)\equiv \hat R(ib)$ for $|b|\le r$.
\item[(b)] $\tilde R(b)\equiv \hat R(0)$ and $\tilde\lambda(b)\equiv1$ for $|b|\ge 2$.
\item[(c)] $\|\tilde R(b)-\hat R(0)\|_\B<\delta$ for all $b\in\R$.
\item[(d)] 
For all $b\in\R$,
the spectrum of $\tilde R(b)$ consists of $\tilde\lambda(b)$ together with a subset of $\{s:|s-1|\ge 3\delta\}$.
\end{itemize}
\end{prop}

\subsection{Wiener's Lemma
for continuous (not necessarily periodic) functions}
\label{sec-W}

Let $G:\R\to\B$ be operator valued functions, where $\B$ is a Banach space with norm $\|\, \|_\B$.
Let $\hat\A$ be the (non-commutative) Banach algebra of $2\pi$-periodic continuous functions
$G:\R\to\B$ such that their Fourier coefficients $\hat G_n$ are absolutely summable, with norm
$\|G\|_{\hat\A}=\sum_{n\in\Z}\|\hat G_n\|_{\B}$. Let $\hat\A_{\beta+1}=\{G\in\hat\A:\sup_{n\in\Z}\ell(|n|)|n|^{\beta+1}|\hat G_n|<\infty\}$ be the Banach algebra
with norm
$\|G\|_{\hat\A_{\beta+1}}=\sum_{n\in\Z}|\hat G_n|+\sup_{n\in\Z}\ell(|n|)|n|^{\beta+1}|\hat G_n|$.
Recall that $\mathcal{\hat R}$ is the non-commutative Banach algebra of continuous functions
$G:\R\to\B$ such that their Fourier transform $\hat G:\R\to\B$ lies in $L^1(\R)$, with norm $\|G\|_{\mathcal{\hat R}}=\int_{-\infty}^\infty\|\hat G(\xi)\|_{\B}\,d\xi$
and that $\mathcal{\hat R}_{\beta+1}=\{G\in\mathcal{\hat R}:\sup_{\xi\in\R}\ell(|\xi|)|\xi|^{\beta+1}\|\hat G(\xi)\|_{\B}<\infty\}$ is a  Banach algebra with norm
$\|G\|_{{\mathcal{\hat R}_{\beta+1}}}=\int_{-\infty}^\infty\|\hat G(\xi)\|_{\B}\,d\xi+\sup_{\xi\in\R}\ell(|\xi|)|\xi|^{\beta+1}\|\hat G(\xi)\|_{\B}$.

Similar definitions apply to the commutative Banach algebras $\A, \A_{\beta+1},\mathcal{R}, \mathcal{R}_{\beta+1}$
 starting from complex valued functions  $G:\R\to\C$.

\begin{lemma}{~\cite[Lemma 8]{BP42}}
 \label{lem-W}
Let $\beta>0$ and let $G_0,G_1\in\mathcal{\hat R}_{\beta+1}$.
Suppose $G_1$ is compactly supported and that $G_0$ is bounded away from zero on the support of $G_1$.
Then there exists $G_2\in\mathcal{\hat R}_{\beta+1}$ such that
$G_1=G_0 G_2$.
\end{lemma} 
The original~\cite[Lemma 8]{BP42} is stated for a Banach algebra $\mathcal{\hat R}$ of $2\pi$ periodic functions. However, given Lemma~\ref{lem-AR} below (a version of~\cite[Lemma 7]{BP42})
Lemma~\ref{lem-W} follows by the argument used in~\cite[Proof of Lemma 8]{BP42}, which requires~\cite[Lemma 6]{BP42} (which holds with $R'$ there replaced by $\hat\A$ defined here)
and Lemma~\ref{lem-AR} below.

\begin{lemma}
\label{lem-AR}
Let $\epsilon>0$.  Suppose that $G:\R\to\B$ is a continuous function with
$\supp G\subset[-\pi+\epsilon,\pi-\epsilon]$.
Let $H:\R\to\B$ denote the $2\pi$-periodic continuous function such that
$H|_{[-\pi,\pi]}=G|_{[-\pi,\pi]}$.
Then $G\in\mathcal{\hat R}$ if and only if $H\in\hat\A$. Moreover, $f\in\mathcal{\hat R}_{\beta+1}$ if and only if $H\in\hat\A_{\beta+1}$.
\end{lemma}
\begin{proof}
The first part on $\mathcal{\hat R}, \hat\A$ is known: see~\cite[Lemma 7]{BP42} (see also~\cite[Theorem 6.2, Ch. VIII, p.~242]{Katzn} for the standard version with commutative Banach algebras).
The second part on $\mathcal{\hat R}_{\beta+1}, \hat\A_{\beta+1}$, follows by, for instance, the argument of~\cite[Lemma A.3]{MelTer17}; the statement and proof of~\cite[Lemma A.3]{MelTer17}
is in terms of the commutative Banach algebras $\mathcal{ R}_{\beta+1}, \A_{\beta+1}$, 
but everything in~\cite[Proof of Lemma A.3]{MelTer17} holds with $\mathcal{\hat R}_{\beta+1}, \hat\A_{\beta+1}$
instead of  $\mathcal{ R}_{\beta+1}, \A_{\beta+1}$.~\end{proof}

{\bf Acknowledgments:}
The support of EPSRC grant EP/S019286/1 is gratefully acknowledged. I also wish thank the referees for their very useful comments that helped me improve the presentation.

\end{document}